\documentclass{amsart}

\usepackage{latexsym}
\usepackage{verbatim}

\usepackage{amssymb,amsmath,amscd,graphicx,color,enumerate}

\usepackage{hyperref}

 \usepackage[normalem]{ulem}
\renewcommand\sout{\bgroup\markoverwith
 {\textcolor{red}{\rule[0.7ex]{3pt}{1.4pt}}}\ULon}

\usepackage[font=bf]{caption}

\usepackage{microtype}

\usepackage[all]{xy}

\usepackage{stmaryrd}

\usepackage{amsmath}

\usepackage{amssymb}

\usepackage[utf8]{inputenc}
\usepackage[T1]{fontenc}
\usepackage[british,UKenglish,USenglish,english,american]{babel}

\usepackage{hyperref}
\usepackage{amsmath}
\usepackage{amssymb}
\usepackage{amsthm,amscd}
\usepackage{mathrsfs}
\usepackage{amsfonts}

\usepackage{esint}
\usepackage{array}
\usepackage{latexsym}

\usepackage{lmodern}

\usepackage{graphicx}
\usepackage{graphics}
\usepackage{epsfig}
\usepackage{color}
\DeclareGraphicsExtensions{.eps,.pdf,.jpeg,.png}

\usepackage{multirow}

\usepackage[all]{xy}
\usepackage{dsfont}

\usepackage{hyperref}
\usepackage{bigints}

\usepackage{color}
\definecolor{darkgreen}{cmyk}{1,0,1,.2}
\definecolor{m}{rgb}{1,0.1,1}
\definecolor{green}{cmyk}{1,0,1,0}
\definecolor{darkred}{rgb}{0.55, 0.0, 0.0}
\definecolor{test}{rgb}{1,0,0}
\definecolor{cmyk}{cmyk}{0,1,1,0}

\usepackage[normalem]{ulem}
\renewcommand\sout{\bgroup\markoverwith
{\textcolor{red}{\rule[0.7ex]{3pt}{1.4pt}}}\ULon}

\usepackage{cancel}

\def\Ch{\operatorname{Ch}}

\def\End{\operatorname{End}}

\def\End{\operatorname{End}}

\def\id{\operatorname{id}}
\def\Id{\operatorname{Id}}

\def\Hom{\operatorname{Hom}}

\def\Ind{\operatorname{Ind}}

\def\Tr{\operatorname{Tr}}

\def\C{\mathbb C}

\def\N{\mathbb N}
\def\R{\mathbb R}

\def\Z{\mathbb Z}

{\theoremstyle {definition} \newtheorem {definition} {Definition} [section] }

{\theoremstyle {definition} \newtheorem {defi} {Definition} [section] }
{\theoremstyle {plain}  \newtheorem {thm} [defi] {Theorem}}
{\theoremstyle {plain}  \newtheorem {cor} [defi]{Corollary}}
{\theoremstyle {plain} \newtheorem {prop} [defi]{Proposition}}
{\theoremstyle {plain} \newtheorem {lem}[defi] {Lemma}}

{\theoremstyle {definition} \newtheorem {ex}[defi] {Example}}

{\theoremstyle {definition} \newtheorem{remarque}[defi]{Remark}}
{\theoremstyle {definition} }
{\theoremstyle {definition} }
{\theoremstyle {definition}  }
{\theoremstyle {definition} }

\newtheorem{notation}{Notation}[defi]

{\theoremstyle {plain}  \newtheorem* {thm*} {Theorem}}

\def\Hom{{\mathrm{Hom}}}
\def\End{{\mathrm{End}}}

\def\Tr{{\mathrm{Tr}}}

\def\Ch{{\mathrm{Ch}}}
\def\det{{\mathrm{det}}}
\def\exp{{\mathrm{exp}}}
\def\g{{\mathfrak{g}}}
\def\h{{\mathfrak{h}}}

\def\k{{\mathrm{K}}}

\newcommand\maA{\mathcal A}

\newcommand\maE{\mathcal E}
\newcommand\maF{\mathcal F}
\newcommand\maP{\mathcal P}

\newcommand\maK{\mathcal K}

\newcommand\maL{\mathcal L}
\newcommand\maH{\mathcal H}

\newcommand\maU{\mathcal U}
\newcommand\maZ{\mathcal Z}

\title{The index of families of projective operators}

\author[A. Baldare]{Alexandre Baldare}
\email{alexandre.baldare@math.uni-hannover.de}
 \address{Institut für Analysis, Welfengarten 1, 30167 Hannover, Germany}
\urladdr{https://baldare.github.io/Baldare.Alexandre/}


\begin{document}

%
%
%
%
%


\begin{abstract}
Let $1 \to \Gamma \to \tilde{G} \to G \to 1$ be a central extension by an abelian finite group.
In this paper, we compute the index of families of $\tilde{G}$-transversally elliptic operators on a $G$-principal bundle $P$. 
We then introduce the notion of families of projective operators
on fibrations equipped with an Azumaya bundle $\maA$. 
We define and compute the index of such families using the cohomological index formula 
for families of $SU(N)$-transversally elliptic operators. 
More precisely, a family $A$ of projective operators  can be pulled back in a family $\tilde{A}$ of $SU(N)$-transversally
elliptic operators on the $PU(N)$-principal bundle of trivialisations of $\maA$.
Through the distributional index of  $\tilde{A}$, we can define an index for the family $A$ of projective operators and
using the index formula in equivariant cohomology for families of $SU(N)$-transversally elliptic operators, 
we derive an explicit cohomological index formula in de Rham cohomology.
Once this is done, we define and compute the index of families of projective Dirac operators. 
As a second application of our computation of the index of families of 
$\tilde{G}$-transversally elliptic operators on a $G$-principal bundle $P$, 
we consider the special case of a family of $Spin(2n)$-transversally elliptic Dirac operators over 
the bundle of oriented orthonormal frames of an oriented fibration 
and we relate its distributional index with the index of the corresponding family of projective Dirac operators.
\end{abstract}

\maketitle
\tableofcontents

\medskip



\section*{Introduction}

This paper is devoted to an application of the cohomological index theorem shown in \cite{baldare:H} using equivariant cohomology.
In particular, using the main result of \cite{baldare:H}, 
we define a cohomological index for famillies of projective operators 
following \cite{MMS1,MMS2,Paradan:projective}.
Let us recall that in the standard case introduced in \cite{MMS1}, Mathai, Melrose and Singer associated with an elliptic projective operator 
an analytical index and then computed this index by a cohomological formula \emph{à la} Atiyah-Singer \cite{atiyah1963index,Atiyah-Singer:I,Atiyah-Singer:III}.
This setting allows them to introduce a projective Dirac operator $\cancel \partial^+_M$ for any oriented manifold and as expected
they obtained 
$$\operatorname{Ind}_a(\cancel \partial^+_M )=(2i\pi)^{-n}\int_M \hat{A}(TM),$$
see \cite{MMS1}. In \cite{MMS2}, the same authors showed that a projective operator $A$ can be represented by 
a $SU(N)$-transversally elliptic operator $\tilde{A}$ and they showed that the analytical index of the projective operator $A$ can be 
computed as the pairing of the distributional index of $\tilde{A}$ with any smooth function on $SU(N)$ equal to $1$ on a 
neighborhood of $\Id \in SU(N)$. 
Recall that the operator $\tilde{A}$ is obtained by pulling back the operator $A$ to the $PU(N)$-principal bundle 
associated with the Azumaya bundle $\maA \rightarrow M$ considered in the definition of the projective operator $A$, see Section \ref{sec.proj}
and \cite{MMS1,MMS2,MMS3} for more details. Notice that here we have a central extension by an abelian finite group 
\begin{equation}\label{eq.ext.SU(N)}
\xymatrix{1 \ar[r] & \Z_N \ar[r] & SU(N) \ar[r]& PU(N) \ar[r] & 1}.
\end{equation}

In \cite{Paradan:projective}, Paradan considered the general case of a central extension by an abelian finite group $\Gamma$ of a compact group $G$ 
\begin{equation}\label{eq.ext.cent}
\xymatrix{1 \ar[r] & \Gamma \ar[r] & \tilde{G} \ar[r]^\zeta & G \ar[r] & 1,}
\end{equation}
and computed the distributional index of any $\tilde{G}$-transversally elliptic operator acting on a $G$-principal bundle $\maP$.
As shown by Atiyah in \cite{atiyah1974elliptic}, this distributional character is supported in the subset $S$ of $\tilde{G}$ 
of elements $\gamma \in \tilde{G}$ such that $\maP^\gamma \neq \emptyset$. 
Since $\tilde{G}$ acts on $\maP$ through the morphism $\zeta$ it follows that
$S\subset \Gamma$.
This allows Paradan to recover the index formula shown in \cite{MMS1} for projective operators using the Berline-Paradan-Vergne index theorem for 
transversally elliptic operators, see \cite{BV:IndEquiTransversal,paradan2008index}.
In particular, around any point of the support of the distributional index character, the index is given by a Atiyah-Singer formula, 
see \cite[Theorem 4.1]{Paradan:projective}. 
Here we point out that this result is completely similar to the results obtained in \cite{yu1991cyclic} 
and that the result of Paradan \cite{Paradan:projective} generalises the result obtained for projective Dirac operators  in \cite{Yamashita}.

In this paper, we follow Paradan's approach and generalise it to the case of families.
More precisely, we consider a central extension by an abelian finite group as in Equation \eqref{eq.ext.cent} and 
a $G$-principal bundle  $\maP \rightarrow M$ where $M \rightarrow B$ is a fibration of compact manifolds.  
In this context, we compute the index of a family of $\tilde{G}$-transversally elliptic operators along the fibres of $\maP \rightarrow B$.
We obtain the following generalisation of  \cite[Theorem 4.1]{Paradan:projective} to families using the index theorem \emph{à la} Berline-Paradan-Vergne shown in \cite{baldare:H} for families of transversally elliptic operators.

\begin{thm*}
Let $\sigma \in K_{\tilde{G}}(T_G(\maP|B))$, 
we have $\mathrm{Ind}^{\maP|B}_{-\infty}(\sigma)=\sum \limits_{\gamma \in \Gamma} T_\gamma(\sigma)\ast \delta_\gamma$, where
$$T_\gamma(\sigma)=(-2i\pi)^{-\dim M +\dim B}\exp_*\Big(\int_{T(M|B)|B} \Ch_\gamma(\sigma)\wedge\hat{A}(T(M|B))^2\wedge e^\Theta\Big).$$
Here $\Ch_\gamma(\sigma)$ is the twisted Chern character, see Definition \ref{def:Chern:twist} and $e^\Theta$ is the Chern-Weil morphism, see Section \ref{sec.vert.twist.Chern}.
\end{thm*}

We then introduce the notion of families of projective operators by considering the special case 
given by the extension of Equation \eqref{eq.ext.SU(N)}. Following \cite{MMS2}, we define the analytical index of such families using
the corresponding pairing with a smooth function on $SU(N)$ equal to $1$ around $\Id \in SU(N)$ with 
the distributional index defined in \cite{baldare:H}, see also Equation \eqref{def:ind:distrib}.
Once this is done, we show using the previous theorem that the index of a projective family can be computed 
with a cohomological formula \emph{à la} Atiyah-Singer with values in the de Rham cohomology of the base $B$.

 
The paper is divided as follows. We start by recalling standard results about functions and distributions
on compact Lie groups. We then recall briefly the definitions of the equivariant cohomologies used in our computations.
In Section \ref{sec.cohomological.index}, we recall the materials from \cite{baldare:KK,baldare:H} regarding the index of families
of transversally elliptic operators, see also \cite{Benameur.Baldare}. In Section \ref{sec.gen.Paradan}, we prove the main result of 
this paper. In Section \ref{sec.proj} we introduce the notion of families of projective operators and show the corresponding
cohomological index formula. We then investigate the example of families of projective Dirac operators 
on Riemannian fibration of oriented, even dimensional manifolds. 
Finally, in the last section we consider the central extension $1 \rightarrow \mathbb{Z}_2 \rightarrow Spin(2n) \rightarrow SO(2n) \rightarrow 1$ and we assume that our compact fibration $p : M\rightarrow B$ is a Riemannian fibration of oriented, even dimensional manifolds.
Under this assumptions, we can define a family $\cancel{\partial}_{\maF|B}^+$ of $Spin(2n)$-transversally elliptic Dirac operators over the total space $\maF$ of the bundle of oriented orthonormal frames, seen as a fibration $p\circ q : \maF\rightarrow B$
and we show the following result.  

%

\begin{thm*}
The distibutional index $\mathrm{Ind}^{\maF|B}_{-\infty}(\cancel{\partial}_{\maF|B}^+)$ of the family $\cancel{\partial}_{\maF|B}^+$ is given by
$$\mathrm{Ind}^{\maF|B}_{-\infty}(\cancel{\partial}_{\maF|B}^+)=T_{M|B}\ast \delta_{\Id} - T_{M|B} \ast \delta_{-\Id},$$
where $T_{M|B}=(2\pi i )^{-n} \exp_* \int_{M|B} \hat{A}(T(M|B)) \wedge e^{\Theta}$ and $\theta$ is the curvature of the $SO(2n)$-principal bundle
$\maF \rightarrow M$.
In particular, if $\varphi \in C^\infty(Spin(2n))^{Ad(Spin(2n))}$ is a function equal to $1$ around $\Id$ and $0$ around $-\Id$ then 
$$\langle \mathrm{Ind}^{\maF|B}_a(\cancel{\partial}_{\maF|B}^+), \varphi \rangle =(2\pi i)^{-n}\int_{M|B} \hat{A}(T(M|B)) \in \maH^{ev}_{dR}(B).$$
\end{thm*}

As an interesting corollary, we get that the index of the 
corresponding family $\cancel{\partial}_{M|B}^+$ of projective Dirac operators is computed by the last equality of the previous Theorem. 

Our first motivation in this paper is to apply the index theory for families 
of transversally elliptic operators developed in \cite{baldare:KK,baldare:H}. 
Furthermore, the case of actions through a central extension is quite natural 
and therefore deserved to be studied properly for families of transversally elliptic operators. 
As an example, it plays a fundamental role in the study of projective elliptic genera introduced 
by Han-Mathai in \cite{Mathai.genera}. It is now possible to investigate in future work the family counterpart of this work.
When the manifold has a  spin structure it is a well known fact that the non vanishing of 
$\hat{A}(M)$ is an obstruction to the existence of a metric with positive scalar curvature, 
see \cite{lichnerowicz1962laplacien,zhang2017positive,connes1983cyclic} and the references therein.
In \cite{gromov1980classification,gromov1980spin,gromov1983positive}, 
the authors introduced the notion of enlargeable manifolds and
proved several important non-existence theorems,   
see also \cite{BenameurHeitschInventiones} where extensions to foliations were considered.
Since then this questions were studied in the context of projective operators in \cite{PennigThese,pennig2014twisted}.
The present work should allow to investigate the same questions in the context of fibrations.\\

We would like to mention that other directions have been investigated in
 \cite{MMS4,MMS3,BenameurGorokhovsky1,BenameurGorokhovsky2,RouseCalcul,RouseTwistedLong} and the references therein. 
In \cite{MMS4,MMS3}, the authors deal with projective families of operators. 
In this case, the twist comes from the base space of the fibration and they obtained 
an index theorem in twisted $K$-theory and then deduced a cohomological formula.
In \cite{BenameurGorokhovsky1}, Benameur and Gorokhovsky showed a local index 
formula for projective families of Dirac operators using Bismut's superconnection approach \cite{bismut}, see also \cite{Quillen:superco}. 
In \cite{RouseTwistedLong}, Carrillo Rouse and Wang extended the setting from \cite{MMS3} 
to the case of foliations and showed a twisted index theorem in $K$-theory.
In \cite{RouseCalcul}, Carrillo Rouse defined the pseudodifferential calculus 
that corresponds to the twisted $K$-theory for Lie groupoids.
Independently, in \cite{BenameurGorokhovsky2} Benameur, Gorokhovsky and Leichtnam 
defined the corresponding pseudodifferential calculus in the special case 
of foliation, i.e. for the holonomy groupoid and showed higher index formulae
 using Bismut's superconnection approach and extended the result of \cite{BenameurGorokhovsky1}.
We point out that none of this results encompass our setting of 
families of projective operators and therefore in particular the setting of \cite{MMS1,MMS2}. 
We refer to \cite[Section 7.2]{RouseCalcul} for a discussion on this subject.

For interesting results concerning index theory, Lie groups and more generally groupoids, 
we refer the reader to \cite{Benameur:LongLefschetzKtheorie,Benameur:thmFamilleLefschetz,BenameurHeitschlefschetz,
Connes:integration:non:commutative,Connes:surveyFoliation,ConnesSkandalis,heitsch1990lefschetz,heitsch1991,hilsum1987morphismes,
Kasparov:KKindex,monthubert1997indice,yu2006higher} and the references therein. 
In particular, we point out the similar setting of gauge-invariant operators investigated in 
\cite{nistor2002asymptotics,nistor2004index,nistor2015analysis}.

\section{Preliminaries}
 
In this section we gather some well known facts about compact Lie groups that we will use in the sequel. 
 
\subsection{Standard applications of Poincaré–Birkhoff–Witt theorem}

This subsection is devoted to standard results  related with Poincaré–Birkhoff–Witt theorem, see for example \cite{bourbakiChap1AlgLie}.
Let $H$ be a compact connected Lie group and $\mathfrak{h}$ its Lie algebra. 
Recall that $H$ acts on itself on the right by $R_g(x)=xg^{-1}$, on the left by $L_g(x)=gx$ 
and therefore by conjugation $Ad(g)x=R_gL_g(x)=L_gR_g(x)$. 
The action by conjugation is called the adjoint action. 
We denote the induced action of an element $s\in H$ on $ \h$ again by $Ad(s)$.\\

Let $\maU(\mathfrak{h})$ denote the universal enveloping algebra 
and $\maZ(\mathfrak{h}):=\maZ(\maU(\mathfrak{h}))$ its center. 
Denote by $\maU(\mathfrak{h})^{\mathfrak{h}}:=\{u\in \maU(\mathfrak{h})\ |\ uX=Xu,\ \forall X\in \mathfrak{h}\}$. 
We clearly have $ \maZ(\mathfrak{h}) \subset \maU(\mathfrak{h})^{\mathfrak{h}}$ 
and similarly if $v\in \maU(\mathfrak{h})^{\mathfrak{h}}$ 
then for any $Y_1,\cdots , Y_k \in \mathfrak{h}$ 
we have $[v,Y_1\cdots Y_k]=0$ and 
such products $Y_1\cdots Y_k$ generate $\maU (\mathfrak{h})$. 
In other words, $\maZ(\mathfrak{h})=\maU(\mathfrak{h})^\mathfrak{h}$. 
We denote by $C^{-\infty}_\gamma(H)$ the set of distributions on $H$ supported in $\gamma$. 
Let $S(\mathfrak{h})$ be the symmetric algebra. The following results are well known, 
we will only give the main ideas of the proofs for the convenience of the reader. 

\begin{prop}\label{prop:PBW}\ 
\begin{enumerate}
\item The enveloping algebra $\maU(\mathfrak{h})$ can be canonically identified 
with the algebra $C^{-\infty}_1(H)$ of distributions on $H$ supported at the identity. 
\item The center $\maZ(\mathfrak{h}) $ corresponds to the set $C^{-\infty}_1(H)^{Ad (H)}$ 
of $Ad (H)$-invariant distributions on $H$ supported at the identity.
\item Let $\gamma \in \maZ(H):=\{h\in H\ | \ \forall t\in H,\ ht=th\}$. 
The map $\maZ(\mathfrak{h}) \rightarrow C^{-\infty}_\gamma(H)^{Ad (H)}$ given by 
$T\mapsto T\ast \delta_\gamma$, where $\delta_\gamma$ is the Dirac delta function in $\gamma$, is an isomorphism.
\item The exponential map $\exp :\mathfrak{h} \rightarrow H$ defines a linear isomorphism (but not of algebras)
$$\exp_* : S(\mathfrak{h})^{Ad(H)} \rightarrow \maZ(\mathfrak{h}),$$
where $S(\mathfrak{h})^{Ad (H)}$ is viewed as the algebra of $Ad$-invariant distributions on $\mathfrak{h}$ supported at $0$.
\end{enumerate}  
\end{prop}   

\begin{proof}
Recall that $C^{-\infty}_1(H)$ is an algebra for the convolution defined by 
$T\ast T'(f)=T\otimes T' (\mu^* f)$, where $\mu : H\times H \rightarrow H$ 
is the product on $H$, i.e. $\mu^*$ is the comultiplication. 
In other words, $T\ast T'(f)=\langle T_{h_1},\langle T'_{h_2},(R_{h_2}f)(h_1)\rangle \rangle$. 
\\

1. Denote by $\delta_1$ the Dirac delta function in $1\in H$. Let $D_1: \mathfrak{h} \rightarrow C^{-\infty}_1(H)$ 
be the map given by $X \mapsto D_1(X):= X^*_H \delta_1$, 
where $X^*_H\delta_1(f)=-X^*_H(f)(1)$ is the derivative of $\delta_1$ along $X\in \mathfrak{h}$. 
Clearly, $D_1([X,Y])=D_1(X)\ast D_1(Y)-D_1(Y)\ast D_1(X)$. 
Therefore, the universal property of $\maU(\mathfrak{h})$ implies that
$D_1$ can be extended to the universal enveloping algebra. 
The map $D_1$ is injective since, by Poincaré-Birkhoff-Witt theorem, 
a basis of $\maU(\mathfrak{h})$ is given by products $X_1^{j_1}\cdots X_n^{j_n}$, 
where $X_i$ is a basis of $\mathfrak{h}$ and $j_i\geq 0$. 
Moreover, the images are linearly independent differential operators composed with the Dirac delta function. 
The surjectivity follows from \cite[Theorem XXXV p 100]{schwartz:distrib}.\\

2. Since $H$ is compact and connected every element $h$ is in the image of the exponential map. 
Therefore, $u\in \maZ(\mathfrak{h})$ if and only if it commutes with every $X\in \mathfrak{h}$.
But this is equivalent to $Ad(e^X)u^*\delta_1=u^*\delta_1$ for any $X\in \mathfrak{h}$.
In other words, $u^*\delta_1$ is $Ad(H)$-invariant.\\

3. Let $T\in \maZ(\mathfrak{h})$ be identified with its corresponding element in $C^{-\infty}_1(H)^{Ad (H)}$. 
Clearly the convolution by $\delta_\gamma$ as support $\gamma$ since $\delta_1\ast \delta_\gamma$ as support $\gamma$. 
Moreover, the convolution by $\delta_\gamma$ is an isomorphism since the convolution by $\delta_{\gamma^{-1}}$ is an inverse. 
Now since $\gamma$ is central we have $Ad(h)\gamma=\gamma$. 
Therefore we get 
$T\ast \delta_\gamma(Ad( h)f)=T\otimes \delta_\gamma(\mu^*(Ad(h)f))=T(Ad(h)R_{Ad(h)\gamma}  f)=T(R_\gamma f)=T\ast \delta_\gamma(f)$.\\

4. Let $v_1,\cdots v_n$ be a basis of $\mathfrak{h}$.
 We can see $S(\mathfrak{h})$ as the algebra of distributions on $\mathfrak{h}$ supported in $0$  
 using the map $\sum a_\alpha v^\alpha \mapsto \sum (-1)^\alpha a_\alpha (v^*)^\alpha \delta_0$,
  where $\alpha=(\alpha_1,\cdots ,\alpha_k)$ and $v^\alpha =v_1^{\alpha_1} \cdots v_k^{\alpha_k}$. 
  The isomorphism is given on monomial by 
  $\exp_*(X_{i_1}\cdots X_{i_p})=\frac{1}{p!}\sum_{\sigma \in S_n} X_{i_{\sigma(1)}}\cdots X_{i_{\sigma(p)}}$ 
  and clearly if $T\in S(\mathfrak{h})$ is $Ad(H)$-invariant then its image sits in $\maZ(\mathfrak{h})$. 
  Indeed, $Ad(e^{tX})\exp_*(T)=\exp_*(Ad(e^{tX})T)=\exp_*(T)$ and 
  therefore $\exp_*(T)$ commutes with every $X\in \mathfrak{h}$, i.e. $\exp_*(T)\in \maZ(\mathfrak{h})$. 
  Notice that the convolution on $S(\mathfrak{h})$ is commutative since it comes from the additive structure on $\mathfrak{h}$. \\ 
Recall the identifications of $S(\mathfrak{h})^{Ad(H)}$ with the algebra of $Ad(H)$-invariant distributions on $\mathfrak{h}$ supported in $0$ 
and of $\maZ(\mathfrak{h})$ with the algebra of $Ad(H)$-invariant distributions  on $H$ supported in $1$. 
The map $\exp_*$ is the usual pushforward of distributions, i.e. if $T \in C^{-\infty}(U)$,
$\phi : U \rightarrow V$ is a smooth map such that $f\vert_{\operatorname{supp}(T)}$ is proper, 
and $f\in C_c^\infty(V)$ then $\phi_*T(f)=T(\phi^*f)$.

\end{proof}

\begin{remarque}
If $\xymatrix{1 \ar[r] & \Gamma \ar[r] &  \tilde{G} \ar[r]^{\zeta} & G \ar[r] &1}$ 
is a central extension of a group $G$ by a finite group $\Gamma$, as in Equation \eqref{eq.ext.cent}, 
then $Ad (G)$-invariant functions and distributions on $\tilde{G}$ 
(respectively on $\g$) are the same as $Ad(\tilde G)$-invariant functions and distributions on $\tilde{G}$ 
(respectively on $\g$). Indeed, let $\tilde{g}_1$ and $\tilde{g}_2 \in \tilde{G}$ 
be such that $\zeta(\tilde{g}_1)=\zeta(\tilde{g}_2)$ then there is $\gamma \in \Gamma$ 
such that $\tilde{g}_1=\tilde{g}_2\gamma$. Now since $\Gamma$ is in the center of $\tilde{G}$, 
we obtain that $Ad(\tilde{g}_1)=Ad(\tilde{g}_2)$ on $\tilde{G}$ and $\g$.\\
In this case we obtain that $\exp_* : S(\mathfrak{g})^{Ad (G)} \rightarrow \maZ(\g)$ 
is an isomorphism, where $\maZ(\g)$ is viewed as $Ad (\tilde{G})$-invariant distributions on $\tilde{G}$ supported in $1$.
\end{remarque}

\subsection{Restrictions of generalized functions}
Here we recall facts about restrictions of generalized functions, see \cite[Section 2.3]{DV:comoEquiDescente}.
Let $H$ be a compact (non necessarily connected) Lie group and let $s\in H$. 
Recall that $H(s):=\{h\in H,\ hs=sh\}$ can be seen as the closed subgroup of 
$H$ given by the stabilizer $\operatorname{Stab}_{Ad (H)}(s)=\{h\in H, \ Ad(h)s=s\}$ in $H$ 
of $s$ for the adjoint action but also as the submanifold of $H$ of fixed points 
$H^{Ad (s)}=\{h\in H, \ Ad(s)h=h\}$ by $Ad(s)$. Denote by $\h(s):=\{Y\in \h,\ Ad(s)Y=Y\}$ 
the Lie algebra of $H(s)$. If we chose a $Ad (H)$-invariant scalar product on $\h$ 
then we get a bi-invariant metric on $H$, i.e. a metric which is both left invariant and right invariant. 
Using this metric $Ad(H)s\times\h(s)$ can be identified with the orthogonal to the $Ad (H)$ orbit of $s$. 
By the slice theorem we obtain that there is an open set $\maU_s(0)\subset \h(s)$ 
such that $H\times_{H(s)} \maU_s(0)$ identifies with an open neighbourhood $W(s,0)$ of $Ad(H)s$. 
The identification $\Phi$ is given by $[k,Y] \mapsto Ad(k)\exp_s(Y)=kse^Yk^{-1}$ 
since the exponential map for a bi-invariant Riemannian metric coincides with the Lie group exponential map. \\
When $V$ is a finite dimensional vector space, we denote by $\det_{V}(J)$ the determinant of $J\in \End(V)$. 
Since $\h(s)$ is $Ad(s)$-invariant, we can restrict $Ad(s)$ to $\mathfrak{q}(s)=\h(s)^\perp\cong\h/\h(s)$. 
Let us recall briefly the following results, see \cite[Section 2.2 $\&$ 2.3]{DV:comoEquiDescente} for more details.

\begin{lem}\cite{DV:comoEquiDescente} Let $s\in H$. 
\begin{enumerate}
\item We have $\h(s)^\perp=\mathrm{im}(\id-Ad(s))$ and 
$\det_{\mathfrak{q}(s)}(\id - Ad(s))>0$.
\item If $\maU_s(0)$ is a small enough neighbourhood of $0\in \h(s)$ then $\det_{\mathfrak{q}(s)}(\id - Ad(se^Y))>0$, $\forall Y\in \maU_s(0)$.
 \item The differential of $\Phi : H\times_{H(s)} \maU_s(0) \rightarrow H$, i.e. the differential  $ d_{[k,Y]}\Phi $ is 
 given modulo composition with $dL_{ke^{-Y}s^{-1}k^{-1}}$ by 
$$D(X,Z)=
Ad(k)\big(e^{-Y}d\exp(Y)Z +(Ad(se^Y)^{-1} -\id)X\big)$$
for any $(X,Z)\in \mathfrak{q}(s) \times \h(s)\cong T_{[k,Y]}H\times_{H(s)} \maU_s(0)$. 
\item We have 
$$|\det(d_{[k,Y]}\Phi)|=\det_{\mathfrak{q}(s)}(\id-Ad(se^Y)) |\det_{\h(s)}(e^{-Y}d\exp(Y))|.$$
\end{enumerate}

\end{lem}  

\begin{proof}
1. Let $v\in \h$ then $v-Ad(s)v \in \h(s)^\perp$. Indeed, let $w\in \h(s)$ then $Ad(s)w=Ad(s^{-1})w=w$ therefore
$$\langle v-Ad(s)v,w\rangle=\langle v,w\rangle -\langle v,Ad(s)^{-1}w\rangle =\langle v,w\rangle - \langle v,w\rangle =0.$$
But $\h(s)=\ker(\id-Ad(s))$ and $\dim \h(s)^\perp=\dim \h -\dim \h(s) = \dim \mathrm{im}(\id-Ad(s))$. 
Recall that $t \mapsto \det_\h(\id - tAd(s))$ is a real polynomial. Now, since $Ad(s) $ is orthogonal,  
every eigenvalue has modulus $1$ and on $\mathfrak{q}(s)=\h/\h(s)$ every eigenvalue is different from $1$. 
If $-1$ is an eigenvalue then $1-(-1)=2>0$. 
Moreover, every complex eigenvalue $\mu$ is pair conjugate 
and $(1-\bar{\mu})(1-\mu)=|1-\mu |^2>0$.\\

2. The map $Y\mapsto f(Y)=\det_{\mathfrak{q}(s)}(\id - Ad(se^Y))$ is continuous and $\det_{\mathfrak{q}(s)}(\id - Ad(s))>0$ therefore 
$f^{-1}(]0 ,+\infty[)$ is open in $\h(s)$ and contains $0$.\\

3. Notice that the tangent space $T_{(0,Y)}\mathfrak{q}(s) \times \h(s)$ to $\mathfrak{q}(s) \times \h(s)$ 
in $(0,Y)$ identifies with the tangent space $T_{[k,Y]}H\times_{H(s)}\maU_s(0)$ 
to $H\times_{H(s)}\maU_s(0)$ in $[k,Y]$. 
Indeed, use the differential of the map 
$j :\mathfrak{q}(s) \times \h(s) \rightarrow H\times_{H(s)} \maU_s(0)$ 
given by $(X,Y) \mapsto [ke^X,Y]$. Furthermore, the left translation gives the identification
$T_{kse^Yk^{-1}}H \rightarrow \h$. Modulo this identifications, 
we have for $(X,Z) \in \mathfrak{q}(s) \times \h(s)\cong T_{(0,Y)} \mathfrak{q}(s) \times \h(s)$: 
$$\begin{array}{lll}
D(X,Z)&=\dfrac{d}{dt}_{|t=0} ke^{-Y}s^{-1}k^{-1}ke^{tX}se^{Y+tZ}e^{-tX}k^{-1}\\
&= Ad(k) \dfrac{d}{dt}_{|t=0} e^{-Y}s^{-1}e^{tX}se^{Y+tZ}e^{-tX}\\
&=Ad(k)\big(Ad(e^{-Y}s^{-1})X +e^{-Y}d\exp(Y)Z -X\big)\\
&= Ad(k)\big(e^{-Y}d\exp(Y)Z +(Ad(se^Y)^{-1} -\id)X\big).
\end{array} $$

4. We have $e^{-Y}d\exp(Y)Z\in \h(s)$ since $e^{Y+tZ}\in H(s)$ and similarly $(\id-Ad(se^Y))X \in \mathfrak{q}(s)$ 
because $(\id-Ad(se^Y))\mathfrak{q}(s) \subset \mathfrak{q}(s)$. 
Since the metric on $\h$ is $Ad(H)$-invariant, we get that $|\det_\h(Ad(k))|=|\det_{\mathfrak{q}(s)}(Ad(se^Y))|=1$.
This gives the result using 3.
\end{proof}


 
%
%
%

Recall that the differential $d\exp(Y)$ of the exponential map 
$\exp : \h(s) \rightarrow H(s)$ is given in $Y\in \h(s)$  by $d\exp(Y)=e^Y\dfrac{1-e^{-ad(Y)}}{ad(Y)}$, 
where $ad$ is the differential of the adjoint action $Ad$ on $\h(s)$. 
We denote as usual by $j_{\h(s)}(Y)=\det_{\h(s )}\bigg( \dfrac{1-e^{-ad(Y)}}{ad(Y)}\bigg)$ the Jacobian determinant of $d\exp(Y)$ 
which is positive on $\maU_s(0)$ if $\maU_s(0)$ is small enough.\\
 
Let $\mu_s$ and $\mu$ denote respectively the normalized Haar measures on $H(s)$ and $H$. 
Recall that there is a unique $H$-invariant measure $\mu_{H/H(s)}$ on $H/H(s)$ 
such that $d\mu=d\mu_s d\mu_{H/H(s)}$ defined by the linear functional 
$f\in C(H/H(s))\mapsto \int_{H/H(s)}f d\mu_{H/H(s)}:=\int_H f\circ \pi(h) d\mu(h)$, where $\pi : H \rightarrow H/H(s)$ is the canonical fibration. We denote $d\mu_{H/H(s)}$ by $dq$ and $d\mu_s$ by $dy$. Denote by $dX$ the tangent Lebesgue measure on $\h$ and respectively by $dY$ and $dQ$ Lebesgue measures on $\h(s)$ and $\mathfrak{q}(s)$ tangent to $dy$ and $dq$ such that $dX=dYdQ$, see \cite{DV:comoEquiDescente}. 

\begin{thm}\cite{DV:comoEquiDescente}\label{thm:descente:Harish:DufloVergne}
Let $\alpha  \in C^\infty(H)^{Ad H}$ and $\varphi \in C^{\infty}(H)$ be functions supported in $W(s,0)\cong H\times_{H(s)} \maU_s(0)$. Then 
$$\int_{W(s,0)} \alpha(h)\varphi(h)d\mu(h)=\int_{H/H(s)} \int_{\maU_s(0)} \alpha(se^Y)\varphi(qse^Yq^{-1})\det_{\mathfrak{q}(s)}(\id - se^Y) j_{\h(s)}(Y) dYdq.$$
Let $\theta \in C^{-\infty}(H)^{Ad H}$ then there is a unique $\psi\in C^{-\infty}(\maU_s(0))^{H(s)}$ such that $\forall \varphi \in C^\infty(H)^{Ad H}$ supported in $W(s,0)$
$$\int_{W(s,0)} \theta(h)\varphi(h)d\mu(h)=\int_{H/H(s)} \int_{\maU_s(0)} \psi(se^Y)\varphi(qse^Yq^{-1})\det_{\mathfrak{q}(s)}(\id - se^Y) j_{\h(s)}(Y) dYdq.$$

\end{thm}

This means that $\theta\in C^{-\infty}(H)^{H}$ defines by restriction an element $\psi\in C^{-\infty}(\maU_s(0))^{H(s)}$. 
We will denote the restricted element by $\theta\|_s$. 
For details on restrictions of invariant generalized functions see for instance \cite{DV:comoEquiDescente, paradan2008index}.

\subsection{Equivariant cohomology}\label{section:equiv:form}
Here we recall the definition of equivariant cohomologies 
used in the sequel, see \cite{BGV,BV:equiChernCaracter,DV:comoEquiDescente}. 
Let again $H$ be a compact Lie group and $\mathfrak{h}$ its Lie algebra. 
Assume that $H$ acts smoothly on a manifold $W$ (we say that $W$ is a $H$-manifold). 
Let $X\in \mathfrak{h}$ and denote by $X_W$ the vector field generated by $X$ on $W$ 
that is $X_W(f)(w)=\frac{d}{dt}_{|t=0}f(e^{-tX}\cdot w)$, $\forall f\in C^{\infty}(W),\ w\in W$. 
Let $d$ be the de Rham differential and let $\iota (Y) $ denote the contraction by a vector field $Y$. 
Let $\mathcal{A}(W)$ be the space of differential forms on $W$. 
Recall that the group $ H$ acts on $\maA(W)$ and consider 
the tensored product $C^{\infty}(\mathfrak{h}) \otimes \mathcal{A}(W)$ equipped
 with the tensored action given by $(s \cdot \alpha)(X)=s(\alpha(Ad(s)^{-1}X))$, 
 for any $\alpha \in C^{\infty}(\mathfrak{h}) \otimes \mathcal{A}(W)$.
Let $\mathcal{A}^{\infty}_H(\mathfrak{h},W)$ denote the algebra 
$\big(C^{\infty}(\mathfrak{h}) \otimes \mathcal{A}(W)\big)^H$ 
of $H$-invariant smooth functions on $\mathfrak{h}$ with values in $\mathcal{A}(W)$.
Let $D$ be the equivariant differential on $\mathcal{A}^{\infty}_H(\mathfrak{h},W)$ given by 
$$(D\alpha)(X)=d(\alpha(X))-\iota (X_W)(\alpha(X)).$$
We have $(D^2\alpha)(X)=-\mathscr{L}(X)\alpha(X)$ so $D^2$ is zero 
on $\mathcal{A}^{\infty}_H(\mathfrak{h},W)$ 
because any element of $\mathcal{A}_H^{\infty}(\mathfrak{h},W)$ is $H$-invariant.

\begin{defi}
The equivariant cohomology $\mathcal{H}^{\infty}_H(\mathfrak{h} ,W)$ with 
smooth coefficients is the cohomology of the complex $(\mathcal{A}^{\infty}_H(\mathfrak{h} ,W),D)$.
\end{defi}

We now recall the definition of the equivariant cohomology with generalised coefficients 
\cite{duflo1990orbites}, see also \cite{kumar1993equivariant}. 
Let $C^{-\infty }(\mathfrak{h},\mathcal{A}(W))$ 
be the space of generalised functions on $\mathfrak{h}$ with values in $\mathcal{A}(W)$. 
By definition, this is the space of continuous linear maps 
from the space $\mathcal{D}(\mathfrak{h})$ of $C^{\infty}$ densities 
with compact support on $\mathfrak{h}$ to $\mathcal{A}(W)$, 
where $\mathcal{D}(\mathfrak{h})$ 
and $\mathcal{A}(W)$ are equipped with the $C^{\infty}$ topologies. 
Therefore, if $\alpha \in C^{-\infty }(\mathfrak{h},\mathcal{A}(W))$ 
and if $\phi \in \mathcal{D}(\mathfrak{h})$ 
then $\langle \alpha ,\phi \rangle $ is a differential form on $W$ denoted by $\int_\mathfrak{h} \alpha(X)\phi(X)dX$. 
A $C^\infty $ density with compact support on $\mathfrak{h}$ is also called a test density, 
and a $C^{\infty}$ function with compact support on $\mathfrak{h}$ is called a test function. 
Denote by $E^i$ a basis of $\mathfrak{h}$ and $E_i$ its dual basis. 
Let $d$ be the operator on $C^{-\infty }(\mathfrak{h},\mathcal{A}(W))$ defined by 
$$\langle d\alpha , \phi \rangle =d\langle \alpha ,\phi \rangle, ~\forall \phi\in \mathcal{D}(\mathfrak{h}).$$
Let $\iota$ be the operator defined by
$$\langle \iota \alpha ,\phi \rangle = \sum\limits_i \iota(E^i_W)\langle \alpha ,E_i \otimes \phi \rangle ,$$
where $E^i_W$ means as before the vector field generated  
by $E^i\in \mathfrak{h}$ on $W$ and $(E_i \otimes \phi)(X)=E_i(X)\phi(X)=X_i\phi(X)$, for any $X=\sum X_i E^i \in \h$.
Let then $d_\mathfrak{h}$ be the operator on $C^{-\infty }(\mathfrak{h},\mathcal{A}(W))$ defined by 
$$d_\mathfrak{h}\alpha=d\alpha - \iota \alpha .$$
The operator $d_\mathfrak{h}$ coincides with the equivariant differential on $C^{\infty }(\mathfrak{h},\mathcal{A}(W))\subset C^{-\infty }(\mathfrak{h},\mathcal{A}(W))$. The group $H$ acts naturally on $C^{-\infty }(\mathfrak{h},\mathcal{A}(W))$ by $\langle g\cdot \alpha , \phi \rangle = g \cdot \langle \alpha ,g^{-1} \cdot \phi \rangle $ and this action  commutes with the operators $d$ and $\iota $. The space of $H$-invariant generalized  functions on $\mathfrak{h}$ with values in $\mathcal{A}(W)$ is denoted by 
$$\mathcal{A}^{-\infty}_H(\mathfrak{h} , \mathcal{A}(W))=C^{-\infty }(\mathfrak{h},\mathcal{A}(W))^H.$$
The operator $d_\mathfrak{h}$ preserves $\mathcal{A}^{-\infty}_H(\mathfrak{h},W)$ and satisfies $d_\mathfrak{h}^2=0$.
Similarly, if we replace $\mathcal{A}(W)$ with $\mathcal{A}_c(W)$ the space of compactly supported forms then we can define $\mathcal{A}^{-\infty}_{c,H}(\mathfrak{h},W)=C^{-\infty}(\mathfrak{h},\mathcal{A}_c(W))^H$. \\
We also need to consider $H$-equivariant generalized forms which are defined on an open neighbourhood of the origin in $\mathfrak{h}$.
If $O$ is an $H$-invariant open subset of $\mathfrak{h}$, we denote by $\mathcal{A}^{-\infty }_{H}(O,W)$ and $\mathcal{A}^{-\infty }_{c,H}(O,W)$ the spaces obtained similarly.  
Let $U$ be a $H$-invariant open subset of $W$.  The space of forms with generalized coefficients and with support in $U$ is denoted by $\mathcal{A}^{-\infty}_U(O,W)$. This is the space of differential forms with generalized coefficients such that there is a $H$-invariant closed subspace $C_\alpha \subset U$ such that $\int \alpha (X)\phi (X) dX$ is supported in $C_\alpha$ for any test density $\phi$.

\begin{notation}
The cohomology of the complex $(\mathcal{A}^{-\infty}_H(\mathfrak{h},W) ,d_\mathfrak{h})$ is denoted by $\mathcal{H}^{-\infty}_H(\mathfrak{h},W)$.\\ 
The cohomology of the complex $(\mathcal{A}^{-\infty}_{c,H}(\mathfrak{h},W) ,d_\mathfrak{h})$ is denoted by $\mathcal{H}^{-\infty}_{c,H}(\mathfrak{h},W)$.\\
The cohomology of the complex $(\mathcal{A}^{-\infty}_{H}(O,W) ,d_\mathfrak{h})$ is denoted by $\mathcal{H}^{-\infty}_{H}(O,W)$.\\
The cohomology of the complex $(\mathcal{A}^{-\infty}_{c,H}(O,W) ,d_\mathfrak{h})$ is denoted by $\mathcal{H}^{-\infty}_{c,H}(O,W)$.\\
The cohomology of the complex $(\mathcal{A}^{-\infty}_{U}(O,W) ,d_\mathfrak{h})$ is denoted by $\mathcal{H}^{-\infty}_{U}(O,W)$.\\
Let $F \subset W$ be a closed subspace and let $\mathcal{H}^{-\infty}_{F}(\mathfrak{h}(s),W^s)$ 
be the projective limit of the projective system $(\mathcal{H}^{-\infty}_U(\mathfrak{h}(s),W^s))_{F\subset U}$.
\end{notation}

\noindent
There is a natural map $$\mathcal{H}^{\infty}_H(\mathfrak{h},W)\rightarrow \mathcal{H}^{-\infty}_H(\mathfrak{h},W)$$ induced by the inclusion $\mathcal{A}^{\infty}_H(\mathfrak{h},\mathcal{A}(W)) \hookrightarrow \mathcal{A}^{-\infty}_H(\mathfrak{h},\mathcal{A}(W))$.
If $p : M \rightarrow B$ is a oriented $H$-equivariant fibration, then integration along the fibres  
$\int_{M|B}$ defines a map from $\mathcal{A}^{-\infty}_{c,H}(\mathfrak{h},M)$ to $\mathcal{A}^{-\infty}_{c,H}(\mathfrak{h},B)$:
$$\langle \int_{M|B}\alpha , \phi \rangle := \int_{M|B}\langle \alpha ,\phi \rangle, ~\forall \phi \in \mathcal{D}(\mathfrak{h}),$$ and induces a well defined map:
$$\int_{M|B} :  \mathcal{H}^{-\infty}_{c,H}(\mathfrak{h},M)\rightarrow \mathcal{H}^{-\infty}_{c,H}(\mathfrak{h},B).$$
Finally note that if $\alpha \in \mathcal{H}^{\infty}_{c,H}(\mathfrak{h},M)$, and $\beta \in \mathcal{H}^{-\infty}_{c,H}(\mathfrak{h},B)$ then $\alpha \wedge p^* \beta \in \mathcal{H}^{-\infty}_{c,H}(\mathfrak{h},M)$ and
$$\int_{M|B}\alpha \wedge p^*\beta =(\int_{M|B}\alpha)\wedge \beta .$$

\section{The index of transversally elliptic families}\label{sec.cohomological.index}
In this section, we first recall the setting of \cite{baldare:KK} and refer to it for details. 
Then we describe the support of the distributional index of families of $H$-transversally elliptic operators 
introduced in \cite{baldare:H}. 
Let $H$ be a compact Lie group and let $p : Z \rightarrow B$ be a compact $H$-fibration with trivial action on $B$. 
We denote by $Z_b=p^{-1}(b)$ the fibre sitting above $b\in B$. We denote by $T(Z|B) := \ker(dp)$ the vertical subbundle of $TZ$.  
Let $\mathfrak{h}$ be the Lie algebra of $H$ 
and recall that an element $X\in \mathfrak{h}$ defines a vector field 
$X_Z(x):=\frac{d}{dt}_{|t=0} e^{-tX} x$ and that $X_Z(x) \in T_x(Z|B)$ is vertical. 
Using a $H$-invariant Riemannian metric on $Z$, we identify $T^*Z$ and $TZ$.

Let us recall the definition of the vertical $H$-transversal cotangent space $T_H^*(Z|B)$. 
Following \cite{atiyah1974elliptic}, we denote by 
$T^*_HZ :=\{(x,\alpha) \in T^*Z,\ \alpha(X_Z(x))=0,\ \forall X\in \mathfrak{h}\}$ 
and identify it with the set $T_HZ$ of vectors orthogonal 
to the orbits with the help of the $H$-invariant Riemannian metric. 
Similarly, we can consider $T^*_H(Z|B) :=\{(x,\alpha) \in T^*(Z|B),\ \alpha(X_Z(x))=0,\ \forall X\in \mathfrak{h}\}$ 
and we can identify it with the set $T_H(Z|B)$ of vertical tangent vectors orthogonal 
to the orbits using the $H$-invariant Riemannian metric. 
We then call  $T_H(Z|B):=T(Z|B) \cap T_HZ$ the vertical $H$-transversal tangent space.

Let $E=E^+ \oplus E^-\rightarrow Z$ be a $\Z_2$-graded hermitian vector bundle. 
In the sequel, we shall denote by $\Psi^m(Z|B,E)$ the set of smooth families of order $m$ pseudodifferential operators on $Z$ 
and by $\Psi^{-\infty}(Z|B,E)$ the smoothing families, see \cite{Atiyah-Singer:IV}.

We shall say that a $H$-invariant smooth family 
$A_0:=\big(A_{0,b} : C^\infty(Z_b, E^+_b) \rightarrow C^\infty(Z_b, E^-_b)\big)_{b\in B}$ 
of pseudodiferential operators is $H$-transversally elliptic if 
its principal symbol $\sigma(A_0)(\xi)$ is invertible for any non zero vector $\xi\in T_H(Z|B)$, see \cite{baldare:KK}. 
Recall that every element $a\in K_{H}(T_H(Z|B))$ of the compactly supported $H$-equivariant $K$-theory group 
of $T_H(Z|B)$ can be represented 
by the principal symbol $\sigma(A_0)_{|T_H(Z|B)}$ of a $H$-invariant family $A_0$ of $H$-transversally elliptic operators. 
Let $A_0^*$ be the formal adjoint of $A_0$ and denote by $A:=\begin{pmatrix}
0&A_0^*\\
A_0&0
\end{pmatrix}$. 
We denote by $\maE=\maE^+\oplus \maE$ the Hilbert $C(B)$-module associated 
with the continuous field $\big(L^2(Z_b,E_b; \mu_b)\big)_{b\in B}$ of square integrable sections 
along the fibres with respect to a $H$-invariant continuous family of Borel measures $(\mu_b)_{b\in B}$ in the Lebesgue class. 
When $A_0$ is a family of order $0$ pseudodifferential operators, 
$A$ extends to an adjointable operator in $\mathcal{L}_{C(B)}(\maE)$. 
Let $C^*H$ be the $C^*$-algebra of the compact group $H$. 
Recall the representation $\pi : C^*H \rightarrow \mathcal{L}_{C(B)}(\mathcal{E})$ 
of $C^*H$ as adjointable operators on $\maE$ given by 
$\pi(\varphi)s(x) = \int_{H} \varphi(h) h(s(h^{-1} x)) dh$, 
where $\varphi \in C(H)$, $s\in C(Z,E)$ and the integration is with respect to the Haar measure on $H$. 
We shall denote the Kasparov's bivariant $K$-theory group  
of the pair of $C^*$-algebras $(C^*H,C(B))$ by $KK(C^*H,C(B))$, see \cite{Kasparov:KKtheory,Kasparov1988}.  

\begin{defi}\cite{baldare:KK}
The analytical index map 
$$\mathrm{Ind}^{Z|B} : K_{H}(T_H(Z|B)) \rightarrow KK(C^*H,C(B))$$
is defined by 
$$\mathrm{Ind}^{Z|B}([\sigma(A_0)_{|T_H(Z|B)}]):=[\maE,\pi,A].$$
\end{defi}

Denote by $\hat{H}$ the set of isomorphism classes of unitary irreducible representations of $H$. 

\begin{prop}
Let $H$ be a compact Lie group and $Z\rightarrow B$ be a compact $H$-fibration with trivial action on $B$. 
Then the analytical index map is a $R(H)$-morphism in the following sense 
$$\mathrm{Ind}^{Z|B}(a\otimes [V])=j^H[V] \underset{C^*H}{\otimes} \mathrm{Ind}^{Z|B}(a), $$
where $V\in \hat{H}$ and $[V]\in K_H(\C)$ is the corresponding element. 
\end{prop}

\begin{proof}
This is exactly the multiplicative property shown in \cite{baldare:KK} with $Z'=\{\star\} \rightarrow B'=\{\star\}$ and $H'=\{1\}$, see also \cite{Benameur.Baldare}.
Let us recall briefly the proof in this simpler case for the benefit of the reader. The index class
$\mathrm{Ind}^{Z|B}(a\otimes [V])$ is represented by $[V \otimes \maE , \pi_V ,  \id_V \otimes A]$,
where $\pi_V(\varphi) (v\otimes \eta)=\int_H \varphi(h) hv \otimes h\eta dh$, $\forall \varphi \in C(H)$, $v\in V$ and $\eta\in \maE$.
The Kasparov product
$j^H[V] \underset{C^*H}{\otimes} \mathrm{Ind}^{Z|B}(a)$ is represented by $[V\rtimes H \underset{\pi}{\otimes} \maE , \rho \otimes_{\pi} 1  , 1 \otimes_{\pi} A]$, 
where $V\rtimes H$ is the completion of $C(H,V)$ with respect to the $C^*H$-valued scalar product given by
$$\langle v_1, v_2 \rangle(h):=\int_H \langle v_1(k) , v_2(kh) \rangle dk,\ \qquad \forall v_1,v_2 \in C(H,V),$$
and $\rho (\varphi) v(h)=\int_H \varphi(k) k(v(k^{-1}h)) dk$, $\forall \varphi \in C(H)$ and $v \in C(H,V)$.
Notice that the operator $1\otimes_\pi A$ is well defined because $[\pi(C^*H),A]=0$ since $A$ is $H$-invariant. 
We then have a unitary equivalence between this two Kasparov modules given by the map
$U : V\rtimes H \underset{\pi}{\otimes} \maE \rightarrow V\otimes \maE$
defined by 
$$U(v \otimes \eta ) =\int_H v(k) \otimes k \eta dk, \ \forall v\in C(H,V), \eta \in \maE.$$
We can easily check that $U(v\cdot \varphi \otimes \eta)=U(v \otimes \pi(\varphi)\eta)$, $\forall v\in C(H,V), \varphi \in C(H),\eta \in \maE$.
Furthermore, for $v_1,v_2 \in C(H,V)$ and $\eta_1,\eta_2 \in \maE$, the identity 
$$\langle U(v_1 \otimes \eta_1),U(v_2 \otimes \eta_2) \rangle=\langle \eta_1 , \pi(\langle v_1,v_2\rangle) \eta_2 \rangle$$
can be checked as follows. Using the $G$-invariance of the scalar product on $\maE$, we have
\begin{align*}
\langle U(v_1 \otimes \eta_1),U(v_2 \otimes \eta_2) \rangle
&=\int_{H^2} \langle v_1(k),v_2(h)\rangle \langle k\eta_1 ,h \eta_2 \rangle dkdh\\
&=\int_{H^2} \langle \eta_1, \langle v_1(k) , v_2(h)\rangle k^{-1}h \eta_2 \rangle dkdh.
\end{align*}
The substitution $u=k^{-1}h$ gives directly
\begin{align*}
\langle U(v_1 \otimes \eta_1),U(v_2 \otimes \eta_2) \rangle
&= \int_{H^2} \langle \eta_1, \langle v_1(k) , v_2(ku)\rangle u \eta_2 \rangle dkdu\\
&=\int_H \langle \eta_1, \langle v_1 , v_2 \rangle (u) u\eta_2 \rangle du\\
&=\langle \eta_1 , \pi(\langle v_1,v_2\rangle) \eta_2 \rangle.
\end{align*}
To show that $U(C(H,V) \otimes_{\pi} \maE)$ is dense in $V \otimes \maE$, 
consider an approximate identity $(e_i)$ of $C^*H$, composed of
continuous functions on $H$ which are supported as close as we please to the neutral element of $H$.  
Then $U((v \otimes e_i) \otimes \eta)=v \otimes \pi(e_i)\eta$ converges to $v\otimes \eta$ for any 
$v\in V$ and $\eta \in \maE$. Similar computations also imply that $U$ intertwines operators and 
representations.
\end{proof}

Recall that Green-Julg isomorphism $K_H(\C) \cong KK(\C,C^*H)$ \cite{julg1981produit:croise} is given by 
$\theta\in \hat{H} \mapsto \chi_\theta = \chi_0\underset{C^*H}{\otimes} j^H([\theta]) \in KK(\C,C^*H)$, 
where $\chi_0=[\C,0] \in KK(\C,C^*H)$ and the Hilbert $C^*H$-module structure is given by 
$\langle\lambda ,\lambda'\rangle (g) =\bar{\lambda} \lambda'$ and $\lambda\cdot \varphi =\lambda \int_H\varphi(h)dh$,
where $\lambda,\lambda' \in \C$ and $\varphi \in C(H)$. 

Let $V\in \hat{H}$ and consider the Hilbert $C(B)$-module $\maE_V^H:=(V\otimes \maE)^H$ and the operator $A_V^H:=(\id_V \otimes A)\vert_{\maE_V^H} \in \maL_{C(B)}(\maE_V^H)$.
We can now introduce the definition of $\k$-multiplicity of an irreductible unitary representation of $H$ from \cite{baldare:KK}. 

\begin{defi}\cite{baldare:KK}
The $\k$-multiplicity $m_A(V)$ of a irreducible unitary representation $V$ of $H$ in the index class  
$\mathrm{Ind}^{Z|B} (A_0)$ is  the image of the class 
$[(\mathcal{E}_V^{H},A_V^{H})]\in KK(\mathbb{C},C(B))$ under the isomorphism 
$KK(\mathbb{C},C(B))\cong K(B)$. 
So $m_A(V)$ is the class of a virtual vector bundle over $B$, 
an element of the topological $\k$-theory group $K(B)$. 
The class $[(\mathcal{E}_V^{H},A_V^{H})]$ coincides (as expected) with the Kasparov product
$$
\chi_V\underset{C^*H}{ \otimes}\mathrm{Ind}^{Z|B}(A_0)\in KK(\mathbb{C},C(B)),
$$
where $\chi_V=\chi_0 \otimes j^H[V] \in KK(\C,C^*H)$ is the element image of $[V]\in K_H(\C)$ by the Green-Julg isomorphism.
\end{defi}

Since $KK(C^*H,C(B))\cong\mathrm{Hom}(R(H),K(B))$ (see for instance \cite{rosenberg1987kunneth}), 
we have the following description of the index map:

\begin{prop}\cite{baldare:H}
The index class of a $H$-invariant family $A_0$ of $H$-transversally elliptic operators is totally determined by its multiplicities and we have:
$$\mathrm{Ind^{M|B}}(A_0)=\sum\limits_{V\in \hat{H}}m_A(V)\chi_V.$$ 
\end{prop}

The next proposition explains that the index map is a $R(H)$-module homomorphism, using the description of the index map from the previous proposition.

\begin{prop}
For any $a\in K_H(T_H(Z|B))$, we have 
$$\mathrm{Ind}^{Z|B}(a\cdot [V])=\sum \limits_{W\in \hat{H}} m_a(W)\chi_V\chi_W.$$
\end{prop}

\begin{proof}
Let $\theta \in \hat{H}$.  We have $\langle \mathrm{Ind}^{Z|B}(a\cdot [V]) , \chi_\theta \rangle = m_{a\otimes V}(\theta) \cong \chi_\theta \underset{C^*H}{\otimes} j^H([V])\underset{C^*H}{\otimes} \mathrm{Ind}^{Z|B}(a)$. Using Green-Julg isomorphism, it follows
$$\begin{array}{lll}\langle \mathrm{Ind}^{Z|B}(a\cdot [V]) , \chi_\theta \rangle
&= \chi_0\underset{C^*H}{\otimes} j^H([\theta])\underset{C^*H}{\otimes} j^H([V])\underset{C^*H}{\otimes} \mathrm{Ind}^{Z|B}(a)\\
&= \chi_0\underset{C^*H}{\otimes} j^H\big([\theta]\underset{\C}{\otimes} [V]\big)\underset{C^*H}{\otimes} \mathrm{Ind}^{Z|B}(a)\\
&=\chi_{\theta\otimes V}\underset{C^*H}{\otimes} \mathrm{Ind}^{Z|B}(a)\\
&= m_a(\theta\otimes V)\\
&=\langle \mathrm{Ind}^{Z|B}(a) , \chi_\theta\chi_V \rangle.
\end{array}$$ 
The last equality follows from the relations
\begin{align*}
\theta \otimes V &=\sum \limits_{W \in \hat{H}
} \dim\big( (W^*\otimes (\theta\otimes V))^H \big) W,\\ 
m_a(\theta \otimes V)&= \sum \limits_{W \in \hat{H}
} \dim\big( (W^*\otimes (\theta\otimes V))^H \big) m_a(W),\\
\langle \chi_W ,\chi_V\chi_\theta\rangle&= 
 \dim\big( (W^*\otimes (\theta\otimes V))^H \big).
\end{align*} 
 
\end{proof}

Let $C^{-\infty}(H)^{Ad(H)}$ be the set of $Ad(H)$-invariant distributions on $H$ and $\maH_{dR}^{ev}(B)$ be the even part of the de Rham cohomology. Assume $B$ oriented. It is shown in \cite{baldare:H} that there is a well defined map 
$$\mathrm{Ind}^{Z|B}_{-\infty} : K_{H}(T_H(Z|B)) \rightarrow C^{-\infty}(H)^{Ad(H)}\otimes \maH_{dR}^{ev}(B)\cong \maH^{-\infty,ev}_H(\h,B)$$
called the distributional index map given by
\begin{equation}\label{def:ind:distrib}
\mathrm{Ind}^{Z|B}_{-\infty}([\sigma(A_0)_{|T_H(Z|B)}])=\sum \limits_{V\in H} \Ch(m_A(V)) \chi_V,
\end{equation}
where $\Ch(m_A(V))\in \maH_{dR}^{ev}(B)$ is the usual Chern character of $m_A(V)$ and $\chi_V$ is the character of $V\in \hat{H}$.\\

We have the following generalisation of \cite[Theorem 4.6]{atiyah1974elliptic}.
\begin{lem}[localisation]\label{lem:localisation}
Let $H$ be a compact Lie group and $ Z\rightarrow B$ be a compact $H$-fibration with $B$ a $H$-trivial oriented manifold.
If $A_0$ is a family of $H$-transversally elliptic operators on $Z\rightarrow B$ then
$$\mathrm{supp} (\mathrm{Ind}^{Z|B}_{-\infty}(A_0))\subset \{h\in H, Z^h\neq \emptyset\} .$$
\end{lem}

\begin{proof}
The proof follows exactly the same line than Atiyah's proof \cite[Theorem 4.6]{atiyah1974elliptic}.
Let $Stab_H(Z)$ be the finite set of conjugacy classes of isotropy subgroup of $H$ for the action on $Z$. 
Let $h\in H$. If $Z^h=\emptyset$ then $h$ is not conjugate to any element belonging in $K\in Stab_H(Z)$. 
Therefore by \cite[Lemma 4.5]{atiyah1974elliptic}, there is $[V] \in K_H(\C)$ such that $\chi(h)\neq 0$ and $\chi_{|K}=0$, 
for any $K\in Stab_H(Z)$. Using \cite[Lemma 4.4]{atiyah1974elliptic}, 
we obtain $[V]^NK_H(Z)=0$ but $K_H(T_H(Z|B))$ is a unitary module on $K_H(Z)$ 
therefore $[V]^NK_H(T_H(Z|B))=0$. Since $\mathrm{Ind}^{Z|B}$ is a $R(H)=K_H(\C)$-homomorphism, 
it follows that $0= \mathrm{Ind}^{Z|B}(\chi^N\cdot a)= [V]^N\cdot \mathrm{Ind}^{Z|B}(a)$ 
and the same is true for the distributional index. 
Since $\chi^N(h)\neq 0$ this implies  $h\notin \mathrm{supp} (\mathrm{Ind}^{Z|B}_{-\infty}(A_0))$.
\end{proof}

\subsection{The Berline-Paradan-Vergne form of the index map for families}

Here we recall the main result of \cite{baldare:H}.  
We will not insist on the construction of the Chern character used in \cite{baldare:H} to proved the index theorem.
This is justified by the fact that in the sequel the vertical transversal space will define a vector bundle. 

Let us denote by $r : T^*(Z|B) \hookrightarrow T^*Z$ the inclusion induced by the Riemannian metric.
The Liouville $1$-form $\omega_Z$ on $T^*Z$ defines by restriction a $1$-form $r^*\omega_Z$ on $T^*(Z|B)$,
see \cite{baldare:H} for more details.
Assume $B$ oriented and $H$-trivial. It can be shown that the $1$-form $r^*\omega $ is $H$-invariant 
and that the subspace $C_{r^*\omega_Z }=\{\xi  \in T^*(Z|B),\ \langle r^*\omega_Z(\xi),X_{T^*(Z|B)}(\xi)\rangle=0, \ \forall X \in \h\}$ of $T^*(Z|B)$ is equal to $T_H^*(Z|B)$, see \cite{baldare:H} for instance. 

Let $\sigma$ be a $H$-transversally elliptic symbol along the fibres of $p : Z\rightarrow B $. 
We recalled above the definition of the distributional index 
$\mathrm{Ind}^{Z|B}_{-\infty}([\sigma ])\in C^{-\infty }(H,\maH^{ev}_{dR}(B))^{Ad(H)}$. 
We can restrict such element through its associated generalized function because such element 
belongs to $C^{-\infty}(H)^{Ad(H)}\otimes \maH^{ev}_{dR}(B)$. 

In the next theorem, we shall denote by $\Ch_{c}(\sigma ,r^*\omega, s)(Y)\in \maH^{-\infty}_{c,H}(\h, T(Z^s|B))$ 
the $s$-equivariant Chern character of a $H$-transversally elliptic morphism along the fibres, see \cite{baldare:H}  and\cite{paradan2008equivariant,paradan2008index} when $B=\ast$.
We denote by  $\hat{A}(T(Z|B),Y)\in \maH^\infty_{H}(\h,Z)$ the equivariant $\hat{A}$-genus of $T(Z|B)$, see \cite{BGV}. 

The main result of \cite{baldare:H} is the following theorem.

\begin{thm}\label{thm:BV:familles}\cite{baldare:H}
Let $\sigma$ be a $H$-transversally elliptic symbol along 
the fibres of a compact $H$-equivariant fibration $p : Z\rightarrow B $ with $B$ oriented and $H$-trivial. 
Denote by $N^s$ the normal vector bundle to $Z^s$ in $Z$. \\
1. There is a unique generalized function with values in the cohomology of $B$ denoted 
$$\mathrm{Ind}^{H,Z|B}_{coh} : K_H(T_H(Z|B)) \rightarrow C^{-\infty}(H , \maH^{ev}_{dR}(B))^{Ad(H)}$$ 
satisfying the following local relations:
$$\mathrm{Ind}^{H,Z|B}_{coh}([\sigma ] )\|_s(Y)=(2i\pi)^{-\dim (Z^s|B)}     \bigint_{T(Z^s|B)}    \hspace*{-0.5cm} \dfrac{\Ch_{c}(\sigma ,r^*\omega, s)(Y)\wedge \hat{A}^2(T(Z^s|B),Y)}{D_s(N^s,Y)},$$
$\forall s\in H$, $\forall Y \in \h(s)$ small enough such that the equivariant classes $\hat{A}^2(T(Z^s|B),Y)$ and $D(N^s,Y)$ are defined. \\
2. Furthermore, we have the following index formula:
$$\mathrm{Ind}^{H,Z|B}_{coh}([\sigma ] )=\mathrm{Ind}^{H,Z|B}_{-\infty}([\sigma ] )\in C^{-\infty}(H ,\maH^{ev}_{dR}(B,\mathbb{C}))^{Ad(H)}.$$
\end{thm}

\begin{remarque}
The definition of the form $D_s(N^s,Y)\in \maH^\infty_H(\h,Z^s)$ can be found in \cite{BV:equiChernCaracter,BV:IndEquiTransversal,paradan2008index} but will not be needed in the sequel since under the assumptions of the next sections $N^s$ will be reduced to $Z \times 0$.
\end{remarque}

Outside of $T^*_HZ$, the $H$-equivariant form $\beta(\omega)=-i\omega \int_0^\infty e^{itD\omega} dt$
is well defined as a $H$-equivariant form with generalized coefficients, and
we have $D\beta(\omega)=1$ outside $T^*_HZ$, see \cite[Equation (15)]{paradan2008index}.
Let $U$ be a $H$-invariant open neighborhood of $T^*_HZ$ and 
let $\chi$ be a smooth $H$-invariant function on $T^*Z$ with support in $U$ 
and equal to $1$ in a neighborhood of $T^*_HZ$. 
Recall \cite[Proposition 3.11]{paradan2008index} that this allows 
to define a closed equivariant differential form on $T^*Z$, 
with generalized coefficients, and supported in $U$
\begin{equation*}
\operatorname{One}(\omega,\chi)=\chi + d\chi \beta(\omega)\in \maA^{-\infty}_U(\g,T^*Z).
\end{equation*}
Moreover, its cohomology class $\operatorname{One}_U(\omega)\in \maH^{-\infty}_U(\g,T^*Z)$ does not depend on $\chi$.

\begin{definition}\cite{paradan2008index}
The collection $(\operatorname{One}_U(\omega))$ defines an element $\operatorname{One}(\omega) \in \maH^{-\infty}_{T^*_HZ,H}(\h,T^*Z)$.
\end{definition}

\begin{remarque}\label{rem.one=1}
If $H=\{e\}$ then $T^*_HZ=T^*Z$ and $\operatorname{One}(\omega)=1$.
\end{remarque}

We denote by $\Ch_{\sup}(\sigma,s)(Y) \in \maH^\infty_{\operatorname{supp} (\sigma),H}(\h,T(Z|B))$ 
the $s$-equivariant Chern character of a vertical symbol $\sigma$ 
defined as in \cite[Definition 3.7]{paradan2008index}, see also \cite{baldare:H}.

\begin{prop}
Let $\sigma$ be a symbol which is $H$-transversally elliptic along the fibres of $Z \to B$. 
We have 
\begin{equation*}
\Ch_{c}(\sigma ,r^*\omega, s)(Y)=\Ch_{\sup}(\sigma,s)(Y) \wedge r^*\operatorname{One}(\omega_s) \in \maH^{-\infty}_{c,H}(\h,T(Z^s|B)),
\end{equation*}
where $\Ch_{c}(\sigma ,r^*\omega, s)(Y)$ is the $s$-equivariant Chern character defined in \cite{baldare:H} using \cite{paradan2008equivariant}.
\end{prop}

\begin{proof}
This follows directly from \cite[Theorem 3.22]{paradan2008equivariant}, see also \cite{paradan2008index}.
\end{proof}

\section{Transversal index for central extension by finite groups}\label{sec.gen.Paradan}

In this section, we generalize the setting from \cite{Paradan:projective} to the context of fibration. 
We recall that $B$ is assumed to be oriented.
Let $p : M\rightarrow B$ be a compact fibration. 
Let $G$ be a compact connected Lie group and $\pi : P \rightarrow M$ be a $G$-principal fibration. 
In particular, we get a compact fibration $p\circ \pi : P\rightarrow B$ and $G$ acts trivially on $B$ as in the previous section. 
As in \cite{Paradan:projective}, we consider a central extension 
$\xymatrix{1 \ar[r] & \Gamma \ar[r] &  \tilde{G} \ar[r]^{\zeta} & G \ar[r] &1}$ by a finite group $\Gamma$. 
In this context, $P \rightarrow B$ becomes a $\tilde{G}$-fibration when equipped 
with the action given by $\tilde{g}\cdot x=\zeta(\tilde{g})\cdot x$, 
for any $x\in P$ and $\tilde{g}\in \tilde{G}$. 
We denote simply by $\tilde{g}\cdot x = \tilde{g}x$ 
and  $g\cdot x = gx$ the actions of $\tilde{G}$ and $G$. \\

We denote by $\mathfrak{g}$ the Lie algebra of $G$ and similarly by $\tilde{\g}$ the Lie algebra of $\tilde{G}$. Notice that $\mathfrak{g}=\tilde{\mathfrak{g} }$ because $\Gamma$ is discrete.
Since the action of $G$ is free on $P$, the map $P\times \g \rightarrow TP$ is an isomorphism on its image. This implies that $T_GP=T_{\tilde{G}}P$ and $T_G(P|B)=T_{\tilde{G}}(P|B)$ are vector subbundles of $TP$. Clearly, the quotient maps by the $G$-action induce  isomorphisms $T_GP/G \cong TM$ and $T_G(P|B)/G \cong T(M|B)$. \\

We are interested in families of $\tilde{G}$-transversally elliptic operators on $P\rightarrow B$.
Let $\sigma\in K_{\tilde{G}}(T_G(P|B))$.
Using Lemma \ref{lem:localisation}, we know that $\mathrm{supp} (\mathrm{Ind}^{P|B}_{-\infty}(\sigma)) \subset \{\tilde{g}\in \tilde{G},\ P^{\tilde{g}}\neq \emptyset\}=\Gamma$. It follows that we can write 
$$\mathrm{Ind}^{P|B}_{-\infty}(\sigma)=\sum\limits_{\gamma\in \Gamma} Q_\gamma(\sigma),$$
where $Q_\gamma(\sigma)\in C^{-\infty}(\tilde{G})^{\tilde{G}}\otimes \maH_{dR}^{ev}(B)$ is supported in $\gamma \in \Gamma$. Using Proposition \ref{prop:PBW}, we obtain that there is $T_\gamma(\sigma) \in \maZ(\mathfrak{g})\otimes \maH_{dR}(B)$ such that $Q_\gamma(\sigma)=T_\gamma(\sigma)\ast \delta_\gamma$, compare with \cite{yu1991cyclic,douglas1995index}. With this in mind, our next goal is to determine $\exp_*^{-1}\otimes \id_{\maH_{dR}(B)}(T_\gamma(\sigma))$.

\subsection{Vertical twisted Chern character}\label{sec.vert.twist.Chern}

Let $E_1,\cdots , E_r$ be an orthonormal basis of $\g$ and 
let $\theta = \sum \theta_i \otimes E_i \in (\maA^1(P) \otimes \g)^G$ be a connection $1$-form on $P\rightarrow M$. 
We denote by $\Theta=\sum \Theta_i \otimes E_i \in (\maA^{ev}_{hor}(P)\otimes \g)^G$ its curvature,
where $\maA^{ev}_{hor}(P)$ is the algebra of horizontal forms of even degree on $P$.
We shall denote by $X_1, \cdots, X_r$ coordinates in the basis $(E_i)$. 
Recall that the Chern-Weil morphism $CW : S(\g)^G \rightarrow \maA(P)^{ev}_{hor}\cong \maA^{ev}(M)$ is given by 
$CW(P)(\Theta)=P(\Theta_1,\cdots, \Theta_r)$ and that 
this can be extended to $C^\infty(\g)^G$ using a Taylor expansion at $0$. 
Let us recall what this means. 
Denote as before $(X^\alpha)^*=(X_1^*)^{\alpha_1}\cdots (X_n^*)^{\alpha_r}$ 
the induced differential operator on $G$ by the monomial $X^\alpha=X_1^{\alpha_1}\cdots X_n^{\alpha_r}$, 
where $\alpha=(\alpha_1,\cdots , \alpha_r)\in \N^r$ is a multi-index. 
Let $\varphi \in C^\infty(\g)^G$ and write 
$\varphi(X_1,\cdots ,X_r)=\sum  \limits_{|\alpha|\leq \dim P} \dfrac{(X^\alpha)^*(\varphi)(0)}{\alpha !} X^\alpha + o(|X|^{\dim P})$, 
with $|\alpha|=\sum \alpha_i$ and $\alpha ! = \alpha_1 !\cdots \alpha_r!$
then $\varphi(\Theta)=\sum \limits_{|\alpha|\leq \dim P} \dfrac{(X^\alpha)^*(\varphi)(0)}{\alpha !} \Theta^\alpha \in \maH_{dR}(M)$.\\
Using the identification of $S(\g)$ with $C^{-\infty}_0(\g)$, see Proposition \ref{prop:PBW}, the Chern-Weil morphism can be written 
$e^\Theta \ast \delta_0$, i.e. $\varphi(\Theta):=\langle e^\Theta \ast \delta_0 , \varphi \rangle_\g$. 
In the sequel, we will denote simply the Chern-Weil morphism by $e^\Theta$ using the previous convention. 

\begin{remarque}\label{rem.phiTheta=1}
If  $\varphi=1$ on a neighbourhood of $0$ then $\varphi(\Theta)=1$. 
\end{remarque}

\begin{defi}\cite{Paradan:projective}\label{def.eTheta} 
For any closed form $\alpha \in \maA_c(T(M|B))$ with compact support, 
the expression $\alpha \wedge e^\Theta$ defines an element in $C^\infty_0(\g) \otimes \maA_c(T(M|B))$. 
Denote by $\overline{\varphi}(X)=\int_G \varphi(Ad(g)X) dg$ the average of $\varphi \in C^\infty(\g)$
 with respect to the Haar measure on $G$. Then $\int_{T(M|B) | B} \alpha \wedge e^\Theta $ defines 
 an element in $C^{-\infty}_0(\g) \otimes \maA(B)$ by 
$$\left\langle \int_{T(M|B) | B} \alpha \wedge e^\Theta , \varphi \right\rangle_\g 
:= \int_{T(M|B) | B} \alpha \wedge\overline{\varphi}(\Theta),$$
for any $\varphi\in C^\infty(\g)$. 
\end{defi}

Following \cite{Paradan:projective} we now introduce the twisted Chern character $\Ch_\gamma(\sigma)$ 
of a $\tilde{G}$-transversally elliptic symbol along the fibres of $P \rightarrow B$. 
Since $\sigma$ is $\tilde{G}$-transversally elliptic along the fibres,
the intersection of its support and $T_G(P|B)$ is compact. 
Seen as a morphism over the manifold $T_G(P|B)$, $\sigma$ is then compactly supported 
therefore the Chern character $\Ch_c(\sigma , \gamma) \in \maH_{c,\tilde{G}}^\infty(\g, T_G(P|B))$ 
is well defined, see \cite{paradan2008equivariant,paradan2008index,Paradan:projective}. 
Since the finite subgroup $\Gamma$ acts trivially on $P$, 
we have a canonical isomorphism between $\maH_{c,\tilde{G}}^\infty(\g, T_G(P|B))$ and $\maH_{c,G}^\infty(\g, T_G(P|B))$. 

\begin{defi}\cite{Paradan:projective}\label{def:Chern:twist}
Let $\maH_{dR,c}(T(M|B))$ denote the de Rham cohomology of $T(M|B)$ with compact support. 
The twisted Chern character $\Ch_\gamma(\sigma) \in \maH_{dR,c}(T(M|B))$ 
is defined as the image of $\Ch_c(\sigma , \gamma)$ under the Chern–Weil isomorphism
$\maH_{c,G}^\infty(\g, T_G(P|B)) \rightarrow \maH_{dR,c}(T(M|B))$
that is associated with the principal $G$-bundle $T_G(P|B) \rightarrow T(M|B)$.
\end{defi}

Let us recall an explicit construction for this Chern character \cite{Paradan:projective}. 
 
\begin{remarque}\label{rem:Paradan:chern}
Let $\Pi : T(P|B) \rightarrow P$ be the projection and $\sigma : \Pi^*E^+ \rightarrow \Pi^*E^-$ 
be a given $\tilde{G}$-transversally elliptic symbol along the fibres.
Let $\nabla^+$ be a $\tilde{G}$-equivariant connection on the vector bundle $E^+\rightarrow P$. 
The pull-back $\nabla^{\Pi^*E^+}:=\Pi^*\nabla^+$ is then a connection on $\Pi^*E^+$ 
viewed as a vector bundle on the manifold $T_G(P|B)$. 
Since $\mathrm{supp}(\sigma)\cap T_G(P|B)$ is compact, 
we can define on the vector bundle $\Pi^*E^- \rightarrow T_G(P|B)$ a connection $\nabla^{\Pi^*E^-}$ such that the relation
$\nabla^{\Pi^*E^-} = \sigma \circ \nabla^{\Pi^*E^+} \circ \sigma^{-1}$ holds outside a compact subset of $T_G(P|B)$.
We consider the equivariant Chern character, twisted by the central element $\gamma \in \Gamma$:
$$\Ch_\gamma^{\tilde{G}}(\sigma):=\Ch_\gamma(\nabla^{\Pi^*E^+}\oplus \nabla^{\Pi^*E^-}),$$
see \cite{BGV,paradan2008equivariant,baldare:H} and the references therein for more details.
\end{remarque}

\subsection{The index formula for central extensions by finite groups}
Let $\theta$ be a connection $1$-form on $\pi : P \rightarrow M$ 
and assume that the metric on $P$ is compatible with the decomposition 
$TP=T_GP \oplus P\times \g$ induced by the connection $\theta$. 
We denote by $\pi_1$ and $\pi_2$ the projections corresponding to the first 
and second factor in the decomposition $TP=T_GP \oplus P\times \g$. 
The differential map $d\pi$ restricted to the subbundle $T_GP$ 
coincides with the quotient map $q : T_GP \rightarrow TM$ by the $G$-action. 
Let $\nu\in \maA^1(P\times \g )^G$ be given by $\nu(x,X)(v,Y)=\langle \theta(x)v , X \rangle_\g$, 
where $(x,X)\in P\times \g$, $(v,Y)\in T_xP \times T_X\g=T_xP\times \g$ 
and $\langle \cdot , \cdot \rangle_\g$ is our metric on $\g$ compatible with the connection and the metric on $P$. 
Let $\omega_P$ and $\omega_M$ be respectively the Liouville $1$-form on $P$ and $M$. Recall that if $Z$ is a manifold 
then the Liouville $1$-form on $T^*Z$ 
is given in local coordinates $(q,p)$ by $\omega=-\sum p_i dq_i$.
In other words $\langle\omega(x,\xi),W\rangle=-\xi(d_x\pi(W))$ 
for any $(x,\xi)\in T^*Z$ and $W\in T_{(x,\xi)}(T^*Z)$. 
%
%
With the notations $d\pi$, $\pi_2$ and the decomposition $T(P|B)=T_G(P|B)\oplus P\times \g$, 
from the previous section, we have the following result.

\begin{prop}\label{cor:decomp:Par}
Assume that the metric on $P$ is compatible with the metrics on $B$ and $M$. 
Denote by $r : T(P|B) \hookrightarrow TP$ the inclusion. Then
$$r^*\omega_P=r^*(d\pi)^*\omega_M -r^*\pi_2^*\nu.$$
Furthermore,
$$r^*\operatorname{One}(\omega_P)=r^*(d\pi)^*\operatorname{One}(\omega_M)\wedge r^*\pi_2^*\operatorname{One}(-\nu)\in \maH^{-\infty}_{\tilde{G},c}(\g,T(P|B)),$$
where $\operatorname{One}(\omega_M)\in \maH_{dR}(TM)$ and $\operatorname{One}(-\nu)\in \maH^{-\infty}_{\tilde{G},c}(\g,P \times \g)=\maH^{-\infty}_{G,c}(\g,P \times \g)$.
\end{prop}
 
\begin{proof}
From \cite[Theorem 4.5]{paradan2008equivariant},
we have
$$\omega_P=(d\pi)^*\omega_M -\pi_2^*\nu,$$
and 
$$\operatorname{One}(\omega_P)=(d\pi)^*\operatorname{One}(\omega_M)\wedge \pi_2^*\operatorname{One}(-\nu),$$
see also \cite[Section 4.1]{paradan2008index}.
The result follows applying the restriction $r^*$.
\end{proof}

\begin{lem}
We have 
$$\hat{A}(T(P|B))^2(X)=(d\pi)^*\hat{A}(T(M|B))^2j_\g(X)^{-1}.$$

\end{lem} 
  
\begin{proof}
Indeed, take on $T(P|B)=q^*(T(M|B))\oplus P\times \g$ the connection given 
by $\nabla^{T(P|B)} =q^*\nabla^{T(M|B)} \oplus d\otimes \id_\g$ where $\nabla^{T(M|B)}$ 
is a connection on $T(M|B)$ and $d$ is the de Rham differential on $P$. 
Then we have 
$$\mu^{T(P|B)}(X) = \maL^{T(P|B)}(X)-\nabla^{T(P|B)}_{X^*_P}=\maL^{P\times \g}(X)-\iota(X)d\otimes \id_\g=\id_P\otimes ad(X),$$
and the curvature of $\nabla^{T(P|B)}$ is $R^{T(P|B)}=q^*R^{T(M|B)}$ 
where $R^{T(M|B)}$ is the curvature of $\nabla^{T(M|B)}$. 
Denoting by $R_{\g}(X)= R^{T(P|B)} +  \mu^{T(P|B)}(X)$, we have by definition 
$$\hat{A}(T(P|B))^2(X)=\det\left(\dfrac{R_\g(X)}{e^{R_\g(X)/2}-e^{-R_\g(X)/2}} \right),$$
see \cite[Section 7.1]{BGV}. The result follows then easily from the relation 
$R_{\g}(X)=q^*R^{T(M|B)} \oplus \id_P\otimes ad(X)$ and the fact that the adjoint action is orthogonal.
\end{proof}  

We shall denote by $\dim(M|B):=\dim M - \dim B$ and $\dim(P|B):=\dim P - \dim B$.
    
\begin{thm}\label{thm.main.paradan}
Let $\sigma \in K_{\tilde{G}}(T_G(P|B))$, 
we have $\mathrm{Ind}^{P|B}_{-\infty}(\sigma)=\sum \limits_{\gamma \in \Gamma} T_\gamma(\sigma)\ast \delta_\gamma$, where
$$T_\gamma(\sigma)=
(2i\pi)^{-\dim (M | B)}\exp_*\Big(\int_{T(M|B)|B} \Ch_\gamma(\sigma)\wedge\hat{A}(T(M|B))^2\wedge e^\Theta\Big).$$
Here $\Ch_\gamma(\sigma)$ is the twisted Chern character, see Definition \ref{def:Chern:twist}.
\end{thm}

\begin{proof} 
Recall that we consider a central extension $\xymatrix{1 \ar[r] & \Gamma \ar[r] &  \tilde{G} \ar[r]^{\zeta} & G \ar[r] &1}$ 
by a finite group $\Gamma$ and therefore 
$\gamma \in \Gamma$ acts trivially on $P$ since $\widetilde{G}$ acts by 
$\tilde{g}\cdot p = \zeta(\tilde{g})p$. 
In particular, we have $P^\gamma=P$, $N^\gamma=P \times \{0\}$ and thus $D_\gamma(N^\gamma,X)=1$. 
We know that $\mathrm{Ind}^{P|B}_{-\infty}(\sigma)$ is supported in $\Gamma$.  
Let $\gamma \in \Gamma$.
Using Theorem \ref{thm:BV:familles}, we have 
$$\mathrm{Ind}^{P|B}_{-\infty}(\sigma)\|_\gamma(X)
=(2i\pi)^{-\dim (P |B)}\int_{T(P|B)|B} \Ch_\gamma(\sigma, X)\wedge r^*\operatorname{One}(\omega_P)\wedge \hat{A}(T(P|B))^2(X).$$
Since $\Ch_\gamma(\sigma ,X)$ is supported in $T_G(P|B)$ 
we have $\Ch_\gamma(\sigma ,X)=\pi_2^*\Ch_\gamma^{\tilde{G}}(\sigma)(X)$ and $r^*(d\pi)^*\operatorname{One}(\omega_M)=1$ because $C_{r^*\omega_M}=T(M|B)$, see Remark \ref{rem.one=1}. 
Therefore applying Corollary \ref{cor:decomp:Par}, we get 
\begin{align*}
\mathrm{Ind}^{P|B}_{-\infty}(\sigma)\|_\gamma(X)
&=(2i\pi)^{-\dim (P |B)}\displaystyle\int_{T(P|B)|B} \Ch_\gamma(\sigma ,X) \wedge r^*\pi_2^*\operatorname{One}(-\nu) \wedge (d\pi)^*\hat{A}(T(M|B))^2 j_\g(X)^{-1}\\
&=(2i\pi)^{-\dim (P |B)}j_\g(X)^{-1}\displaystyle\int_{T_G(P|B)|B} \Ch_\gamma^{\tilde{G}}(\sigma) \wedge (d\pi)^*\hat{A}(T(M|B))^2 \int_\g \operatorname{One}(-\nu). 
\end{align*}
But using \cite[Lemma 4.5]{paradan2008index}, 
$\displaystyle\int_\g \operatorname{One}(-\nu) =(2i\pi)^{\dim G} e^\Theta \ast \delta_0 \theta_r \cdots \theta_1$. 
Therefore, we obtain
$$\mathrm{Ind}^{P|B}_{-\infty}(\sigma)\|_\gamma(X)
=(2i\pi)^{-\dim (M |B)}j_\g(X)^{-1}\int_{T(M|B)|B} \Ch_\gamma(\sigma) \wedge \hat{A}(T(M|B))^2 e^\Theta \ast \delta_0.$$ 
Since $\tilde{G}(\gamma)=\tilde{G}$, the result follows from Theorem \ref{thm:descente:Harish:DufloVergne}.
\end{proof}

\begin{cor}\label{cor:Paradan:famille}
Let $\gamma \in \Gamma$ and $\varphi \in C^\infty(\tilde{G})$ be a function equal to $1$ on a neighbourhood of $\gamma$ with small enough support. Then
$$\langle\mathrm{Ind}^{P|B}(\sigma),\varphi \rangle_{\tilde{G}}=
(2i\pi)^{-\dim (M | B)}\int_{T(M|B)|B} \Ch_\gamma(\sigma)\wedge\hat{A}(T(M|B))^2.$$
\end{cor}
	
\begin{proof}
If the support of $\varphi$ is small enough then the only element of $\Gamma$ contained in the support of $\varphi$ is $\gamma$.
Therefore, Theorem \ref{thm.main.paradan} gives
\begin{equation*}
\langle\mathrm{Ind}^{P|B}(\sigma),\varphi \rangle_{\tilde{G}}=\langle T_\gamma(\sigma)\ast \delta_\gamma , \varphi \rangle_{\tilde{G}},
\end{equation*}
 where $T_\gamma(\sigma)=(-2i\pi)^{-\dim (M | B)}\exp_*\Big(\displaystyle\int_{T(M|B)|B} \Ch_\gamma(\sigma)\wedge\hat{A}(T(M|B))^2\wedge e^\Theta\Big).$
 Since $\varphi$ is equal to $1$ around $\gamma$, we get the result because
  $\langle e^\Theta \ast \delta_0(X) , \varphi(\gamma e^X) \rangle_\g$ is equal to $1$ in cohomology.
\end{proof}	
	
Following \cite{Paradan:projective}, we consider the group $\hat{\Gamma}$ of characters of the finite abelian group $\Gamma$ 
and we decompose any $\tilde{G}$-transversally elliptic symbol along the fibres 
$\sigma\in C^\infty(T(P|B), \mathrm{Hom}(\Pi^*E^+, \Pi^*E^-))$ as $\sigma=\bigoplus_{\chi\in \hat{\Gamma}} \sigma_\chi$, 
where $\sigma_\chi\in C^\infty(T(P|B), \mathrm{Hom}(\Pi^*E^+_\chi, \Pi^*E^-_\chi))$ is a $\tilde{G}$-transversally elliptic symbol 
along the fibres on $P$. Here $E^\pm_\chi$ is the subbundle of $E^\pm$ where $\Gamma$ acts through the character $\chi$. 
From Definition \ref{def:Chern:twist}, it is obvious that the twisted Chern character $\Ch_\gamma(\sigma)$ admits the decomposition
$$\Ch_\gamma(\sigma)=\sum \limits_{\chi\in \hat{\Gamma}}\chi(\gamma) \Ch_e(\sigma_\chi),$$
see also \cite{Paradan:projective}.
We then obtain the following theorem, see again \cite[Theorem 4.3]{Paradan:projective} for the case $B=\ast$.
\begin{thm} 
Let $\sigma \in \k_{\tilde{G}}(T_G(P|B))$ with decomposition $\sigma=\bigoplus_{\chi\in \hat{\Gamma}} \sigma_\chi$. 
We have
$$\mathrm{Ind}^{P|B}_{-\infty}(\sigma)=\sum \limits_{(\chi,\gamma) \in \hat{\Gamma}\times \Gamma} \chi(\gamma)
T_e(\sigma_\chi)\ast \delta_\gamma,$$
where $T_e(\sigma_\chi)=(2i\pi)^{-\dim(M|B)}\exp_*\big(\displaystyle\int_{T(M|B)|B}  \Ch_e(\sigma_\chi)\hat{A}(T(M|B))^2 e^{\Theta} \big)$.
\end{thm}	
	
\begin{proof}
This follows using linearity and Theorem \ref{thm.main.paradan}.
\end{proof}

Let us give now an example.
\begin{ex}
Let $ S^3 \rightarrow \mathbb{CP}^1=S^2$ be the Hopf $S^1$-principal fibration. 
Let $\theta=\frac{i}{2} \sum_{i=0}^1 z_kd\bar{z}_k -\bar{z}_k dz_k$ be the standard 
connection associated with the Fubini-Study metric on $\mathbb{CP}^1$.
Using the connection we can identify $T^*(S^3|S^2)=S^3 \times \R^*$ with the orthogonal 
to the horizontal bundle $H=\ker \theta$, i.e. $T^*(S^3|S^2) \cong \{\xi \in T^*P,\ \xi |_H=0\}$.
Recall that the identification is given by $\eta \to \eta \circ \theta$ for $\zeta \in \R^*$.
Denote by $r : T^* (S^3|S^2)\hookrightarrow T^*S^3$ the induced inclusion and by $\omega_{S^3}$ the Liouville $1$-form.
Then if $\zeta$ denote the coordinate in $\R^*$ and  $\pi^{S^3|S^2}  : T^*(S^3|S^2) \rightarrow S^3$ the bundle projection, 
we get $-r^*(\omega_{S^3})=\zeta(\pi^{S^3|S^2})^*\theta$ which will be denoted simply by $\zeta\theta$ 
since $T^*(S^3|S^2)=S^3\times \R^*$ is trivial.
Indeed, let $(x,\xi)\in T^*(S^3|S^2)=S^3 \times \R^*$ and $W \in T_{(x,\xi)}(T^*(S^3|S^2))$ then
$$-r^*(\omega_{S^3})(x,\xi)(W)=\langle r(\xi) , d\pi_{S^3}(dr(W)) \rangle.$$
Since $\pi \circ r$ is the bundle projection of $\pi^{S^3|S^2}  : T^*(S^3|S^2) \rightarrow S^3$, we get 
\begin{align*}
-r^*(\omega_{S^3})(x,\xi)(W)  &= \langle r(\xi),  d\pi^{S^3|S^2}(W) \rangle\\
& = \xi \circ \theta ( d\pi^{S^3|S^2}(W) ) \\
&=\xi(1) \theta ( d\pi^{S^3|S^2}(W) ).
\end{align*}
Consider now the fibrations $\pi : P=S^3 \times S^1 \rightarrow S^2\times S^1=M$ and $p : M=S^2 \times S^1 \rightarrow S^2=B$.
The Liouville $1$-form on $P$ is given by $p_1^*\omega_{S^3} + p_2^*\omega_{S^1}$, where $p_1 : S^3 \times S^1 \rightarrow S^3$ 
and $p_2 : S^3 \times S^1 \rightarrow S^1$ are the projections. We assume that $S^1$ acts on $P$ by its action on $S^3$. 
Let $\sigma \in K_{S^1}(T^*_{S^1}(P|B))$ and let us compute $\mathrm{Ind}^{P|B}_{-\infty}(\sigma)$. Notice that since 
we consider the free action of $S^1$ the index is supported in $e$. We follow \cite{paradan2008index}.
We have that $\hat{A}(T^*(P|B))^2(X)=1$ since the bundle $T^*(P|B)$ is trivial and the Lie bracket on $\R$ is $0$.
Therefore using that $\operatorname{One}(\omega_{S^1})=1$, we get
\begin{align*}
\mathrm{Ind}^{P|B}_{-\infty}(\sigma)\|_e(X)
&=(2i\pi)^{-2}\int_{T^*(P|B)|B} \Ch_e(\sigma, X)\wedge \operatorname{One}(r^*\omega_P)\\
&=(2i\pi)^{-2}\int_{S^3 \times T^*S^1|S^2} \Ch_e^{S^1}(\sigma) \int_{\R^*}\operatorname{One}(r^*\omega_{S^3}).
\end{align*}


Let us compute $\int_{\R^*}\operatorname{One}(r^*\omega_P)$. Let $g \in C^\infty_c(\R)$ be equal to $1$ on a neighborhood of $0$.
Consider the function $\chi(\zeta)=g(\zeta^2)$ on $P\times \R^*$ and represent $\operatorname{One}(r^*\omega_{S^3})$ by
$$\operatorname{One}(r^*\omega_{S^3})=\chi + d\chi \wedge (-ir^*\omega_{S^3})\int_0^\infty e^{itD(r^*\omega_{S^3})} dt.$$  
Let us denote by $\hat{\phi}(\zeta)=\int_\R e^{i\zeta X} \phi(X) dX$ the Fourier transform 
of any smooth compactly supported function $\phi$ on $\R$. 
We now compute $\int_\R \left(\int_{\R^*} \operatorname{One}(r^*\omega_{S^3}) \phi(X) dX\right)$.
First notice that $D(r^*\omega_{S^3})=d(r^*\omega_{S^3})-\iota(X)(r^*\omega_{S^3})=-(d\zeta \wedge \theta + \zeta \Theta -\zeta X)$,
where $\Theta$ is the curvature of $\theta$. We then have for any compactly supported function $\phi$ on $\R$:
\begin{align*}
\int_{\R} \bigg( \int_{\R^*} \operatorname{One}(r^*\omega_{S^3})\bigg) \phi(X) dX 
&=\int_{\R^*} \int_{\R} d\chi (i \zeta \theta) \int_0^\infty e^{-it(d\zeta \wedge \theta +\zeta \Theta -\zeta X)} dt \phi(X) dX \\
&=\int_{\R^*} \int_{\R} d\chi (i \zeta \theta) \int_0^\infty (1 -it\zeta \Theta )e^{it\zeta X}  \phi(X) dt dX,
\end{align*}
because $\theta\wedge \theta =0$ and $ \Theta$ is a $2$-form on the $2$ dimensional manifold $S^2$.
It follows
\begin{align*}
\int_{\R} \bigg(\int_{\R^*} \operatorname{One}(r^*\omega_{S^3})\bigg) \phi(X) dX
&=\int_{\R^*}  d\chi (i \zeta \theta) \int_0^\infty (1 -it\zeta \Theta ) \hat{\phi}(t\zeta) dt  \\
&=i\theta\int_{\R^*} \int_0^\infty -2\zeta g'(\zeta^2) \zeta  (1 -it\zeta \Theta ) \hat{\phi}(t\zeta) dt d\zeta,
\end{align*}  
where the change of sign is due to $-d\zeta \wedge \theta= \theta \wedge d\zeta$ 
which give the orientation on the fibers of $T^*(S^3|S^2)\rightarrow S^2$ induced from the Liouville $1$-forms of $S^3$ and $S^2$.
We now use the substitution $\zeta_1=t\zeta$ to obtain 
\begin{align*}
\int_{\R} \bigg(\int_{\R^*} \operatorname{One}(r^*\omega_{S^3})\bigg) \phi(X) dX
&=i\theta\int_{\R^*} \int_0^\infty -2 g'(\frac{\zeta^2_1}{t^2}) \frac{\zeta_1}{t^2}  (1 -i\zeta_1 \Theta ) \hat{\phi}(\zeta_1) dt \frac{d\zeta_1}{t}\\
&=i\theta\int_{\R^*} \int_0^\infty  \frac{d}{dt}(g(\frac{\zeta^2_1}{t^2}))  (1 -i\zeta_1 \Theta ) \hat{\phi}(\zeta_1) dt d\zeta_1.
\end{align*}  
Since $g$ is compactly supported and equal to $1$ on a neighborhood of $0$ it follows
\begin{align*}
\int_{\R} \bigg(\int_{\R^*} \operatorname{One}(r^*\omega_{S^3})\bigg) \phi(X) dX
&=i\theta\int_{\R^*}   (1 -i\zeta_1 \Theta ) \hat{\phi}(\zeta_1) d\zeta_1,\\
&=2\pi i\theta (\phi(0)+\phi'(0)\Theta)\\
&=2\pi i\theta \phi(\Theta),
\end{align*} 
where we have used that $ \int_{\R^*} \hat{\phi}(\zeta) d\zeta=2\pi\phi(0)$ and $\int_{\R^*} -i\zeta\hat{\phi}(\zeta) d\zeta=2\pi\phi'(0)$.
Notice that $\theta$ can be integrated along the fibers of $S^3 \rightarrow S^2=\mathbb{CP}^1$ using the trivialisations 
$$t_0 : \C \times S^1 \rightarrow S^3\vert_{U_0},\ (\alpha, \beta) \to \frac{\beta}{\sqrt{1+|\alpha|^2}}(1,\alpha),$$
and
$$ t_1: \C \times S^1 \rightarrow S^3\vert_{U_1},\ (\alpha, \beta) \to \frac{\beta}{\sqrt{1+|\alpha|^2}}(\alpha,1),$$
where $U_i$ are the standard affine charts of $\mathbb{CP}^1$. More precisely, we have above $U_0$
\begin{align*}
\int_{U_0\times S^1|U_0} t_0^*(z_0 d\bar{z}_0-\bar{z}_0dz_0) 
&=\int_0^{2\pi} \bigg(\frac{e^{i\lambda}}{\sqrt{1+|\alpha|^2}}\frac{-ie^{-i\lambda}}{\sqrt{1+|\alpha|^2}}d\lambda
-\frac{e^{-i\lambda}}{\sqrt{1+|\alpha|^2}}\frac{ie^{i\lambda}}{\sqrt{1+|\alpha|^2}}d\lambda \bigg) \\
&=-\frac{4i\pi}{1+|\alpha|^2}, \\
\int_{U_0\times S^1|U_0} t_0^*(z_1 d\bar{z}_1-\bar{z}_1dz_1)
&=\int_0^{2\pi} \bigg(\frac{e^{i\lambda}\alpha}{\sqrt{1+|\alpha|^2}} \frac{-i\bar{\alpha}e^{-i\lambda}}{\sqrt{1+|\alpha|^2}}d\lambda
-\frac{\bar{\alpha}e^{-i\lambda}}{\sqrt{1+|\alpha|^2}}\frac{i\alpha e^{i\lambda}}{\sqrt{1+|\alpha|^2}}d\lambda \bigg)\\
&=-\frac{4i\pi|\alpha|^2}{1+|\alpha|^2},
\end{align*}
and therefore as expected
\begin{align*}
\int_{U_0\times S^1 |U_0 } \theta=\frac{ i}{2}\bigg(\frac{-4i\pi}{1+|\alpha|^2}+\frac{-4i\pi|\alpha|^2}{1+|\alpha|^2}\bigg)=2\pi. 
\end{align*}
Notice that $\bar{\phi} (\Theta)=\int_{S^1} \phi(Ad_z(X)) dz = \phi(\Theta)$ because $S^1$ is commutative.
All together this gives the result.  
\end{ex}

\begin{remarque}
\begin{enumerate}
\item To get concrete example, we can take any elliptic operator of positive order on $S^1$ and
consider it as a constant family of $S^1$-transversally elliptic operators on $S^3\times S^1$. 
\item The same computation for the fibration $q : S^3 \rightarrow S^2$ computes the distributional index 
of the zero family of $S^1$-transversally elliptic operators $0 : C^\infty(S^3) \rightarrow 0$. In this case,
we have that $\operatorname{Ind}^{S^3|S^2}(0)=\sum_{n\in \Z} m_0(\C_n) \chi_n$, where $m_0(\C_n)=S^3 \times_{S^1} V$ and $\chi_n(z)=z^n$ is the character of the irreducible representation $\C_n$.
Therefore, using Chern-Weil isomorphism we get that $\Ch(m_0(\C_n))=\Ch(P\times \C_n)(\Theta)=e^{\mu_n(1)\Theta}$,
where $\mu_n(X)=ni=\mathscr{L}^{\C_n}(X) -X_{S^3}$ is the moment of the trivial connection $d$ on $P\times \C_n$.
Since $\mu_n(1)=in$, we get $\Ch(m_0(\C_n))=1+in \Theta$ and therefore 
$$\Ind^{S^3|S^2}_{-\infty}(0)=\sum_{n\in \Z}(1+in\Theta)\chi_n.$$
In other words, $\forall \phi \in C_c^\infty(\R)$ supported close enough from $0$, we have
\begin{align*}
\langle \Ind^{S^3|S^2}_{-\infty}(0)(e^{iX}) ,\phi(X) \rangle &= \int_{\R} \sum_{n\in \Z} (1+in\Theta) e^{inX} \phi(X)  dX,\\
&=2\pi \sum_{n\in \Z} (\hat{\phi}(n)+\Theta \hat{\phi'}(n)),\\
&=2\pi (\phi(0) + \Theta \phi'(0)),\\
&=\langle \Ind^{S^3|S^2}_{coh}(0)(e^{iX}) ,\phi(X) \rangle.
\end{align*}
\item We can consider a central extension of $S^1$ by $S^1$ given by $z\to z^n$ with kernel $\mathbb{Z}_n$. 
In this case, the computation done in the last example can be reproduced \emph{mutatis mutandis} 
to get the distributional coefficients at $\gamma \in \mathbb{Z}_n$. 
\end{enumerate}

\end{remarque}

\section{The index of families of projective operators} \label{sec.proj}
 
In this section, we extend the setting of \cite{MMS1,MMS2} to the case of families.
Let $\maH$ be a Hilbert space and denote by $\maK(\maH)$ the $C^*$-algebra of compact operators. 
Let us first recall the definitions of Azumaya bundles and projective bundles.
 
\begin{defi}\cite{MMS1,MMS3} An Azumaya bundle $\maA$ over a manifold $M$ is a vector bundle with
fibres which are Azumaya algebras and which has local trivializations reducing
these algebras to $M_N(\C)$. 
A projective vector bundle $E$ over $M$ is a projection valued section of $\maA \otimes \maK(\maH)$. 

\end{defi}


Recall that the transpose Azumaya bundle $\maA^t$ is $\maA$ with multiplication reversed. 
Since the structure group of $\maA\otimes \maA^t$ acts by the adjoint representation $PU(N) \rightarrow PU(N^2)$ 
which lift canonically to a $U(N)$ action, the bundle  $\maA\otimes \maA^t$ is trivial as an Azumaya bundle, see \cite{MMS1}.

\begin{lem}\cite{MMS1}
Let $E_1$ and $E_2$ be projective bundles associated to $\maA$. 
Then the bundle $\mathrm{hom}(E_1,E_2)$ with fibres $\mathrm{hom}(E_{1x},E_{2x})$ at $x\in M$ is a vector bundle.
\end{lem}

Let $\pi : \maP_\maA \rightarrow M$ be the $PU(N)$-principal bundle of trivialisations of $\maA \to M$. 
Then the lift $\pi^*\maA$ of $\maA$ to $\maP_\maA$ is trivial, i.e. it is a homomorphism bundle. 
Let $E_1$ be a projective vector bundle. Then $\tilde{E}_1=\pi^*E_1$ is a finite dimensional vector bundle 
such that $\tilde{E}_1\subset \C^N\otimes \maH$ which is equivariant for the standard action of $U(N)$ on $\C^N$ 
interpreted as covering the action of $PU(N)$ on $\maP_\maA$. 
Let $E_2$ be an other projective vector bundle associated with $\maA$. 
Recall that the action of $U(N)$ on $\mathrm{hom}(\tilde{E}_1,\tilde{E}_2)$ is by conjugation. 
Therefore, $\mathrm{hom}(\tilde{E}_1,\tilde{E}_2)$ defines a $PU(N)$-equivariant vector bundles over $\maP_\maA$ 
which descends to a well defined vector bundle $\mathrm{hom}(E_1,E_2)$ on $M$.\\
Unfortunately, the ``big'' homomorphism bundle $\mathrm{Hom}(\tilde{E}_1,\tilde{E}_2)$ 
is only a projective vector bundle over $M^2=M\times M$ since it is associated with $\maA\boxtimes\maA^t$ over $M^2$. 
By the previous discussion, $\mathrm{Hom}(\tilde{E}_1,\tilde{E}_2)$ restricts to the diagonal
in  a vector bundle, reducing there to $\mathrm{hom}(\tilde{E}_1,\tilde{E}_2)$. \\
Denote by $d$ the distance function associated with the Riemannian metric on $M$. 
Let $$N_\varepsilon : =\{(x,x')\in M^2,\ d(x,x')<\varepsilon\}.$$

Let $p : M \rightarrow B$ be a compact fibration as before. Let us recall the following fondamental result \cite{MMS1}.


\begin{prop}\label{prop.ext.Hom}
Given two projective bundles, $E_1$ and $E_2$, associated to a fixed
Azumaya bundle and $\varepsilon > 0$ sufficiently small, the exterior homomorphism bundle 
$\mathrm{Hom}( \tilde{E}_1, \tilde{ E}_2)$ over $M_p^2:=M\times_B M = \{(x,x') \in M \times M,\ p(x)=p(x')\}$, descends from a neighborhood of the diagonal in $\maP_\maA \times_B \maP_\maA= \{ (z,z')\in \maP_\maA \times \maP_\maA,\ p(\pi(z))=p(\pi(z'))\}$ to a vector bundle, $\mathrm{Hom}^\maA(E_1, E_2)$, over $N_{\varepsilon,B}:=N_\varepsilon \cap M_p^2$ extending $\mathrm{hom}(E_1, E_2)$. For any three such bundles
there is a natural associative composition law
$$ \mathrm{Hom}^\maA_{(x'',x')}(E_2, E_3) \times \mathrm{Hom}^\maA_{(x,x')}(E_1,E_2 )\rightarrow \mathrm{Hom}^\maA_{(x,x'')}(E_1,E_3),$$
given by $(a,a')\mapsto a\circ a'$ for any $(x'',x')$, $(x,x')\in N_{\varepsilon/2,B}$ which is consistent with the composition over the units in $M_p^2$.
\end{prop}

\begin{proof}
It is shown in \cite[Proposition 1]{MMS1} that for $\varepsilon > 0$ sufficiently small, the exterior homomorphism bundle 
$\mathrm{Hom}( \tilde{E}_1, \tilde{ E}_2)$, descends from a neighborhood of the diagonal in $\maP_\maA \times \maP_\maA$ to a vector
bundle, $\mathrm{Hom}^\maA(E_1, E_2)$, over $N_\varepsilon$ extending $\mathrm{hom}(E_1, E_2)$ with the  associative composition law.
The result follows then by restriction to $N_{\varepsilon,B}$.
\end{proof}

Let $F_1$ and $F_2$ be vector bundles over $M$. 
Denote by $|\lambda(M|B)|$ the vector bundle of vertical  densities over $M$ and 
by $|\Lambda(M|B)|$ its pullback to $M_p^2$ through the first projection.  
Recall that families of smoothing operators $\Psi^{-\infty}(M|B,F_1,F_2)$ can be defined 
as operators associated with  smooth kernels $C^\infty(M_p^2,\mathrm{Hom}(F_1,F_2) \otimes |\Lambda(M|B)|)$ 
over $M_p^2=M\times_B M$, i.e. $A\in \Psi^{-\infty}(M|B,F_1,F_2)$ is given by a smooth section 
$A(x,x')\in C^\infty(M_p^2,\mathrm{Hom}(F_1,F_2)\otimes|\Lambda(M|B)|)$ by the formula
$$As(x)=\int_{M_b}A(x,x')s(x'),\qquad s\in C^\infty(M,F_1 )  .$$ 
Furthermore, if $F_3$ is an other vector bundle over $M$ then the composition
$$\Psi^{-\infty}(M|B;F_2,F_3) \circ \Psi^{-\infty}(M|B,F_1,F_2) \subset \Psi^{-\infty}(M|B;F_1,F_3)$$
is given by $$A\circ B(x,x')=\int_{M_b}A(x,x'')\circ B(x'',x') .$$

Following \cite{MMS1}, we now define the linear space of families of smoothing operators and families of pseudodifferential operators with kernels
supported in $N_{\varepsilon,B}$ for any pair $E_1$, $E_2$ of projective bundles associated to a fixed Azumaya bundle.

\begin{defi}
Let $E_1$, $E_2$ be projective bundles associated to a fixed Azumaya bundle $\maA$. 
The linear space of families of smoothing operators with kernel supported in $N_{\varepsilon,B}$ is
\begin{equation*}
\Psi^{-\infty}_\varepsilon (M|B,E_1,E_2): = C^\infty_c(N_{\varepsilon,B},\mathrm{Hom}^\maA(E_1,E_2)\otimes |\Lambda(M|B)|).
\end{equation*}
\end{defi}
%

\begin{prop}\cite{MMS1}
 Let $E_1$, $E_2$ and $E_3$ be projective bundles associated to a fixed Azumaya bundle $\maA$.
The composition law of usual families of smoothing operators can be extended directly to define
$$\Psi^{-\infty}_{\varepsilon/2} (M|B;E_2,E_3) \circ \Psi^{-\infty}_{\varepsilon/2} (M|B;E_1,E_2) \subset \Psi^{-\infty}_\varepsilon (M|B;E_1,E_3).$$
For $A\in \Psi^{-\infty}_{\varepsilon/4} (M|B,E_4,E_3) $, $B\in \Psi^{-\infty}_{\varepsilon/4} (M|B,E_3,E_2) $ and $C\in \Psi^{-\infty}_{\varepsilon/4} (M|B,E_2,E_1) $ this product is associative, i.e.
$$A\circ(B\circ C)=(A\circ B)\circ C.$$
\end{prop}

\begin{proof}
As in \cite{MMS1}, this follows directly from the composition law of Proposition \ref{prop.ext.Hom}.
\end{proof}

\begin{defi}
Let $E_1$, $E_2$ be projective bundles associated to a fixed Azumaya bundle $\maA$. 
The space of families of order $m$ pseudodifferential operators with kernel supported in $N_{\varepsilon,B}$ is
\begin{equation*}
\Psi^{m}_\varepsilon (M|B,E_1,E_2): = I^{m-\frac{\dim B}{4}}_c(N_{\varepsilon,B},M)\underset{C^\infty_c(N_{\varepsilon,B})}{\otimes}  C^\infty_c(N_{\varepsilon,B},\mathrm{Hom}^\maA(E_1,E_2)),
\end{equation*}
where $ I^{m-\frac{\dim B}{4}}(N_{\varepsilon,B},M)$ is the set of compactly supported order ${m-\frac{\dim B}{4}}$ conormal distributions to $M$ on $N_{\varepsilon,B}$, see \cite{Hormander3,Melrose}.
\end{defi}

We have the following standard results, see for example \cite{Hormander1971,Hormander3,Melrose,MMS1,Shubin92}. 
See also \cite{DS2,lauter2000pseudodiff,LescureManchonVassout,NWX,vassout2006unbounded}.

\begin{thm}\cite{MMS1} Let $E_1$, $E_2$ and $E_3$ be projective bundles associated to a fixed Azumaya bundle $\maA$.
\begin{enumerate}
\item Then
$$\hspace*{-0,5cm}\xymatrix{0\ar[r]& \Psi^{m-1}_{\varepsilon} (M|B;E_1,E_2) \ar[r]& \Psi^{m}_{\varepsilon} (M|B;E_1,E_2) \ar[r]^{\hspace*{-1.2cm}\sigma_m}& C^{\infty}(S^*(M|B),\mathrm{hom}(E_1,E_2)\otimes N_m)\ar[r]&0,}$$ 
where $N_{m}$ is the line bundle over $S^*(M|B)$ of smooth functions on $T(M|B)\setminus 0$ 
which are homogeneous of degree $m$.  
\item The composition law of usual families of smoothing operators can be extended directly to define
$$\Psi^{m}_{\varepsilon/2} (M|B;E_2,E_3) \circ \Psi^{m'}_{\varepsilon/2} (M|B;E_1,E_2) \subset \Psi^{m+m'}_\varepsilon (M|B;E_1,E_3).$$
\item   For $A\in \Psi^{m}_{\varepsilon/4} (M|B,E_4,E_3) $, $B\in \Psi^{m'}_{\varepsilon/4} (M[B,E_3,E_2) $ and $C\in \Psi^{m''}_{\varepsilon/4} (M|B,E_2,E_1) $ we have
$$A\circ(B\circ C)=(A\circ B)\circ C.$$
\item Furthermore, the symbol map satisfies 
 $$\sigma_{m+m'}(AB)=\sigma_m(A)\sigma_{m'}(B).$$
\item If $A\in \Psi^{m}_{\varepsilon/2}(M|B;E_1,E_2)$  
is elliptic, i.e. $\sigma_m(A)$ is pointwise invertible on $T(M|B)\setminus 0$, then there exists $Q\in \Psi^{-m}_{\varepsilon/2} (M|B;E_2,E_1)$ 
such that $Q\circ A = \Id-E_R$, 
$A\circ Q = \Id-E_L$, 
where $E_R \in \Psi^{-\infty}_{\varepsilon} (M|B;E_1,E_1)$ and 
$E_L\in \Psi^{-\infty}_{\varepsilon} (M|B;E_2,E_2)$. 
Furthermore, any two such choices $Q'$ and $Q$ satisfy 
$Q'- Q \in \Psi^{-\infty}_{\varepsilon/2} (M|B;E_2,E_1)$.
\end{enumerate}

\end{thm}

Recall the central extension $\xymatrix{1 \ar[r] & \Z_N \ar[r] & SU(N) \ar[r]& PU(N) \ar[r] & 1}$.
The following result is shown in \cite[Proposition 4]{MMS2} in the case $B=\{\ast\}$.  

\begin{thm}\cite{MMS2}\label{thm.MMS2.projective.trans} 
Let $\Omega \subset \maP_{\maA,\pi\circ p}^2$ be a sufficiently small neighborhood of $\maP_\maA$ 
invariant under the diagonal $PU(N)$-action. Then there is a well defined push-forward
map into the families of projective pseudodifferential operators
$$\pi_* : \Psi^m_{\Omega}(\maP_\maA |B;\tilde{E}_1, \tilde{E}_2)^{SU(N)}:=\{A\in \Psi^m(\maP_\maA |B;\tilde{E}_1, \tilde{E}_2)^{SU(N)},\ \mathrm{supp}(A)\subset \Omega\} \rightarrow \Psi^m_\varepsilon(M|B;E_1,E_2)$$
which preserves composition of elements with support in $\Omega'$ such that $\Omega'\circ \Omega' \subset \Omega$.

\end{thm}

\begin{proof}
Notice that if $(x,x')\in M_p^2$ then clearly $\pi^{-1}(x) \times \pi^{-1}(x') \subset \maP_{\maA,\pi\circ p}^2=\maP_{\maA} \times_B \maP_{\maA}$. 
Let $A\in \Psi^m_{\Omega}(\maP_\maA |B;\tilde{E}_1, \tilde{E}_2)^{SU(N)}$ 
and denote by $A_b(p,p')$ the family of Schwartz kernels over $\maP_{\maA,\pi\circ p}^2$.  
We then define the map $\pi_*$ as in \cite{MMS2} by the formula
\begin{equation*}
\pi_*A_b(x,x')=\int_{\pi^{-1}(x) \times \pi^{-1}(x')} A_b(p,p').
\end{equation*}
The rest of the proof is completely similar to \cite[Proposition 4]{MMS2} and is omitted here.
\end{proof}

Let $\pi_M : T(M|B) \rightarrow M$ be the projection. 
As in \cite{MMS1,MMS2}, the symbol $\sigma(A)$ of an elliptic family of projective operators $A$  
defines an element $[\sigma(A)]\in K(T(M|B),\pi_M^*\maA)$ of the
compactly supported twisted K-theory, 
see \cite{bouwknegt2002twisted,donovan,karoubi1968algebres,karoubi2008twisted,rosenberg,tu2004twisted} and the references therein.
Denoting by $\pi_{\maP_\maA} : T_{SU(N)}(\maP_\maA|B) \rightarrow \maP_\maA$ the projection, 
we obtain as in \cite{MMS2} a map in $\k$-theory
$$\iota : K(T(M|B),\pi_M^*\maA) \rightarrow K_{SU(N)}(T_{SU(N)}(\maP_\maA|B)),$$
given by $\iota([\sigma(A)])=[\widetilde{\sigma(A)},\pi_{\maP_\maA}^*\tilde{E}_1, \pi_{\maP_\maA}^*\tilde{E}_2]$ 
where $\widetilde{\sigma(A)} \in C\big(T_{SU(N)}(\maP_{\maA}|B) \setminus \maP_\maA, \hom(\tilde{E}_1,\tilde{E}_2)\big)^{PU(N)}$ 
is the $PU(2^n)$-invariant section corresponding to $\sigma(A) \in C\big(T(M|B)\setminus M,\hom(E_1,E_2)\big)$
with respect to the $PU(N)$-principal bundle $T_{SU(N)} (\maP_\maA|B) \rightarrow T(M|B)$. 
This is $SU(N)$-invariant with respect to the action covering the $PU(N)$-action since the 
action by conjugation does not depend on the representative of a lift of an element of $PU(N)$.

\begin{defi}
Let $A\in \Psi^m_\varepsilon(M|B;E_1,E_2)$ be an elliptic family of projective operators. 
Denote by $\tilde{A}$ the pullback family to $\maP_\maA$. 
Let $\phi\in C^{\infty}(SU(N))$ be a function equal to $1$ in a small enough neighbourhood of the identity. 
Then we define the analytical index of $A$ by 
$$\mathrm{Ind}^{M|B}_a(A)=\sum \limits_{V\in \widehat{SU(N)}} \Ch(m_{\tilde{A}}(V)) \langle \chi_V, \varphi \rangle \in \maH^{ev}_{dR}(B).$$
\end{defi}

We have the following Atiyah-Singer index formula.

\begin{thm}\label{thm.index.proj}
Let $A\in \Psi^m_\varepsilon(M|B;E_1,E_2)$ be an elliptic family of projective operators. 
Then
$$\mathrm{Ind}^{M|B}_a(A)=(2\pi i)^{-\dim(M|B)}\int_{T(M|B)|B} \Ch_e(\iota[\sigma(A)]) \wedge \hat{A}(T(M|B))^2 \in \maH^{ev}_{dR}(B).$$
\end{thm}

\begin{proof}
We apply Corollary \ref{cor:Paradan:famille} to the  central extension
\begin{equation*}
\xymatrix{1 \ar[r] & \Z_N \ar[r] & SU(N) \ar[r]& PU(N) \ar[r] & 1}.
\end{equation*}
\end{proof}

\subsection{Families of projective Dirac operators}\label{section.fam.proj.Dirac}
Assume that $T(M|B)$ is oriented and that $\dim (M|B)=\dim M -\dim B=2n$. 
Recall that $B$ is also assumed to be oriented. 
Consider the special case where the Azumaya bundle $\maA=\C l(M|B)$ is the complexified Clifford bundle of $T(M|B)$. 
Denote as before by $\pi : \maP \rightarrow M$ the $PU(2^n)=Aut(\C l(2n))$-principal bundle of trivializations associated with $\maA$. 
We assume that the metric $g_M$ on $M$ is constructed from the pull back of a metric $g_B$ on $B$ 
and a metric on $T(M|B)$, i.e. $g_M=g_{M|B}\oplus p^*g_B$. 
Notice that this can be achieved by picking a random metric on $M$ 
and replacing the metric on the orthogonal to $T(M|B)$ by $p^*g_B$.
Similarly, we assume that the metric $g_\maP$ on $\maP$ is given by 
$g_\maP = \langle\cdot ,\cdot \rangle_{\mathfrak{su}} \oplus \pi^*g_M$, 
where $\langle \cdot ,\cdot \rangle_{\mathfrak{su}}$ is a metric on $T(\maP|M)$.
  
Let $\maF:=F^{SO}(M|B)$ be the bundle of oriented orthonormal frames of $T(M|B)$. 
We have the identification $\maP=\maF\times_{SO(2n)} PU(2^n)$, where 
$SO(2n) \hookrightarrow PU(2^n)=Aut(\mathbb{C}l(2n))$ is the standard embedding, see \cite{SpinGeometry}. 
 
Following \cite{MMS1}, we proceed now to the construction of the family of projective Dirac operators.
Let us fix from now on a $PU(2^n)$-equivariant $\ast$-isomorphism 
$$\Phi : \mathbb{C}l(2n) \rightarrow M_{2^n}(\mathbb{C}).$$ 
\begin{lem}
\begin{enumerate}
\item We have the following $PU(2^n)$-equivariant trivialisation $T : \pi^*\mathbb{C}l(M|B) \rightarrow \maP \times \mathbb{C}l(2^n)$ given by
$$T(f,\varphi) = (f,f^{-1}(\varphi)),$$
where $f\in \maP$ is seen as an $\ast$-isomorphism $f : \mathbb{C}l(2n) \rightarrow \mathbb{C}l(M|B)_{\pi(f)}$.
\item Since $\maP=\maF \times_{SO(2n)} PU(2^n)$ and $\mathbb{C}l(M|B)=\maF \times_{SO(2n)} \mathbb{C}l(2n)$, 
the previous trivialisation can be rewritten
$$([\mathscr{E},A], [\mathscr{E},\tilde{\varphi}])  \mapsto ([\mathscr{E},A], \hat{A}^{-1} \tilde{\varphi}\hat{A}),$$
where $\hat{A}\in SU(2^n)$ is any lift of $A\in PU(2^n)$.
\item The previous trivialisation induces using $\Phi$ the trivialisation 
$$\pi^*\mathbb{C}l(M|B) \rightarrow \maP \times M_{2^n}(\mathbb{C}).$$
\item We can then define
$$c : \pi^*T(M|B) \rightarrow \maP \times M_{2^n}(\mathbb{C}).$$ 
\end{enumerate}

\end{lem}

\begin{proof}
The action of $PU(2^n)$ on $\maP$ is given by $f \cdot A=f\circ A$ therefore
we clearly have 
$$ T\Big((f,\varphi)\cdot A\Big)=T\Big(f\circ A, \varphi\Big) = \Big(f\circ A , A^{-1}\circ f^{-1}(\varphi)\Big).$$
Since any automorphism of $\mathbb{C}l(2n) \cong M_{2^n}(\mathbb{C})$ is inner,
we can write $T\Big((f,\varphi)\cdot A\Big)=\Big(f\circ A , \hat{A}^{-1} (f^{-1}(\varphi))\hat{A}\Big)$, 
where $\hat{A}\in SU(2^n)$ is any lift of $A \in PU(2^n)$ 
through the central extension $SU(2^n) \rightarrow PU(2^n)$ by $\Z_{2^n}$.\\

The three last items are similar and therefore we shall only explain the last item.
Recall that $T_{SU(2^n)}(\maP |B)=\pi^*TM$ and that any element of $\pi^*TM$ is of the form 
$([\mathscr{E},A],[\mathscr{E},v])\in (\maF\times_{SO(2n)} PU(2^n)) \times_M (\maF \times_{SO(2n)} \R^{2n})$. 
Using the isomorphism $\Phi$ and the trivialisation $\pi^*\mathbb{C}l(M|B) \rightarrow \maP \times \mathbb{C}l(2^n)$,
we can define the map 
$$c : \pi^*T(M|B) \rightarrow \maP \times M_{2^n}(\mathbb{C})$$
by $c([\mathscr{E},A],[\mathscr{E},v])=([\mathscr{E},A], \Phi(\hat{A}^{-1}\hat{c}(v)\hat{A}))$,
where $\hat{c}: \R^{2n} \rightarrow \mathbb{C}l(2n)$ is the standard map and $\hat{A}\in SU(2^n)$ is any lift of $A\in PU(2^n)$.
The map $c$ does not depend on the choices. Indeed, $c$ clearly does not depend on the choice 
of the lift in $SU(2^n)$ because of conjugation and for any lift $\hat{R}\in Spin(2n)$ of $R\in SO(2n)$, 
we have $\hat{c}(R^{-1}v)=\hat{R}^{-1}\hat{c}(v)\hat{R}$ therefore 
$c([\mathscr{E}\circ R,R^{-1}A],[\mathscr{E}\circ R,R^{-1}v])=([\mathscr{E},A], f(\hat{A}^{-1}\hat{R}\hat{c}(R^{-1}v)\hat{R}^{-1}\hat{A}))=c([\mathscr{E},A],[\mathscr{E},v])$.
\end{proof}

Denote by $\omega_\C=i^n \hat{c}(e_1)\cdots \hat{c}(e_{2n})$ the chirality element in $\mathbb{C}l(2n)$ and
let $\tilde{\omega}_\C([\maE,A])=\Phi(\hat{A}^{-1}\omega_\C \hat{A})$ be the corresponding section of $\maP \times M_{2^n}(\mathbb{C})$.
Notice that this is the equivariant section obtained from the chirality global section $\omega_{M|B} : M \rightarrow \mathbb{C}l(M|B)$ 
which at $x\in M$ is given by $\omega_{M|B}(x)=i^n v_1\cdots v_{2n}$ for any local oriented orthonormal basis $(v_i)$ of $ T_x(M|B)$.

Following \cite{MMS1}, let us introduce the projective vertical $Spin$ bundle.
\begin{defi}
We call the $SU(2^n)$-equivariant trivial bundle $\tilde{\mathbb{S}}:=\maP \times \mathbb{C}^{2^n}$
the vertical projective $Spin$ bundle.
The vertical projective half spinor bundle  $\mathbb{S}^\pm$ are the projective bundles 
associated with the projections $\frac{1\pm\omega_{M|B}}{2}$.
They can be represented by the $SU(2n)$-equivariant vector bundles 
$\tilde{\mathbb{S}}^\pm:=\left(\frac{1\pm\tilde{\omega}_\C}{2}\right)\tilde{\mathbb{S}}$. 
\end{defi}

\begin{remarque}
The previous definition coincides with the definition of the projective 
$Spin$ bundle introduced in \cite{MMS1} when $B$ is reduced to a point.
\end{remarque}

As in \cite{MMS1,MMS2}, 
the Levi-Civita connection induces partial connections $\nabla^\pm$ on $\mathbb{S}^\pm$.
More precisely, on $T(M|B)$ we consider the connection $\nabla^{LCV}=\pi^V\circ \nabla^{LC}$, 
where $\pi^V : TM \rightarrow T(M|B)$ is the projection induced by the metric on $M$ 
and $\nabla^{LC}$ is the Levi-Civita connection on $M$. 
Then $\nabla^{LCV}$ induces the Levi-Civita connection on each fibre $M_b$
and the pullback of $\nabla^{LCV}$ is a metric connection on $T_{SU}(\maP|B)=\pi^*T(M|B)$. 
This in turn defines $SU(2^n)$-equivariant connections $\nabla^{\pm}$ on $\tilde{\mathbb{S}}^{\pm}$ 
using the group embedding $\rho : SO(2n) \hookrightarrow PU(2^n)$, 
see \cite{SpinGeometry} for instance.

\begin{lem}
Let $\theta$ be the connection one form on $\maF$ induced by $\nabla^{LVC}$.
\begin{enumerate}
\item The $1$-form $\tilde{\theta} : T\maF \times TPU(2^n) \rightarrow \mathfrak{pu}(2^n) \cong \mathfrak{su}(2^n)$ given by 
$$\tilde{\theta}(\mathscr{E},v,A,X)=Ad_{A^{-1}}\big(\rho_*(\theta(v))\big) + X, \qquad (\mathscr{E},v,A,X)\in T\mathcal{F} \times PU(2^n)\times \mathfrak{pu}(2^n)$$
induces a $SU(2^n)$-invariant connection $\bar{\theta} : T\maP \rightarrow \mathfrak{pu}(2^n)$.
\item The $SU(2^n)$-equivariant connection $\nabla$ on $\tilde{\mathbb{S}}$ is given by $\nabla = d + U_*(\bar{\theta})$, 
where $U_* : \mathfrak{pu}(2^n) \rightarrow M_{2^n}(\mathbb{C})$ is the composition 
of the isomorphism $\mathfrak{pu}(2^n) \rightarrow \mathfrak{su}(2^n)$ and the inclusion $\mathfrak{su}(2^n) \rightarrow M_{2^n}(\mathbb{C})$. 
Then the connections $\nabla^\pm$ on $\tilde{\mathbb{S}}^\pm$ are given by $\nabla^\pm=\frac{1\pm\tilde{\omega}_\C}{2}\nabla$.   
\item The connection $\nabla$ is a Clifford connection, i.e. if $\nabla^{\maP \times M_{2^n}(\mathbb{C})}=\Phi \circ T\nabla^{\pi^*\mathbb{C}l(M|B)}T^{-1} \circ \Phi^{-1}$ is the connection induced by the pullback connection of $\nabla^{LCV}$ modulo the isomorphism $\Phi \circ T$ then
$$\big[\nabla , C \big] = \nabla^{\maP \times M_{2^n}(\mathbb{C})}(C),\ \forall C \in C^\infty(\maP,M_{2^n}(\mathbb{C})).$$ 
\end{enumerate}

\end{lem}

\begin{proof}
1. Denote by $r : SO(2n) \times \maF \times PU(2^n) \rightarrow \maF \times PU(2^n)$ the action, i.e. $r(R,\mathscr{E},A)=(\mathscr{E}\circ R, R^{-1}A)$.
We have
\begin{align*}
\tilde{\theta}\bigg(dr\big((R,Y),(\mathscr{E},v,A,X)\big)\bigg)
&= \tilde{\theta}\bigg(\Big(\mathscr{E}\circ R,d_\mathscr{E}r_R\big(v+Y_\maF(\mathscr{E})\big)\Big),\Big(R^{-1}A, -R^{-1}YA+R^{-1}AX\Big) \bigg)\\
&=Ad_{A^{-1}R}Ad_{R^{-1}}\big(\theta(\mathscr{E},v) + Y\big) + dL_{A^{-1}R}\big((-R^{-1}YA +R^{-1}AX\big)\\
&=\tilde{\theta}\bigg(\Big(\mathscr{E},v\Big),\Big(A,X\Big)\bigg),
\end{align*}
where we have used the identification of $TPU(2^n)\cong PU(2^n)\times \mathfrak{pu}(2^n)$ given by left translation.
Furthermore, for any $g\in SU(2^n)$ we have 
\begin{align*}
\tilde{\theta}\bigg(\Big(\mathscr{E},v\Big), \Big(A,X\Big)\cdot g\bigg)
&= Ad_{g^{-1}}\bigg(Ad_{A^{-1}}\big(\theta(\mathscr{E},v)\big)\bigg) +dL_{g^{-1}A^{-1}}\bigg(AXg\bigg)\\
&=Ad_{g^{-1}}\bigg(\tilde{\theta}\Big(\big(\mathscr{E},v\big),\big(A,X\big)\Big)\bigg).
\end{align*}
Therefore, $\tilde{\theta}$ induces a $SU(2^n)$-invariant connection $\bar{\theta}$ on $\maP$.\\

2. This follows directly from the $SU(2^n)$-invariance of $\bar{\theta}$. 
Indeed, let $s \in \mathbb{C}^{2^n}$, $g \in SU(2^n)$, $v\in C^\infty(P,TP)$ 
and denote by $r : \maP \times SU(2^n) \rightarrow \maP$ the action, i.e. $r(f,g)=f\circ \beta(g)$ where $\beta : SU(2^n) \rightarrow PU(2^n)$ is the quotient map,  then 
\begin{align*}
g\bigg(\nabla_{d_fr_g(v)}s\bigg)(f\circ g)
&= g\bigg(d_{f\circ \beta(g)}s \circ d_fr_g(v)\bigg) + gU_*\bigg(\bar{\theta}(d_fr_g(v))\bigg)s(f\circ \beta(g))\\
&=d_f(g\cdot s)(v) + gU_*\bigg(Ad_{g^{-1}}Big(\bar{\theta}(v)\Big)\bigg)s(f\circ \beta(g))\\
&=d_f(g\cdot s)(v) + U_*\bigg(\bar{\theta}(v)\bigg)(g\cdot s)(f),
\end{align*}
where we recall that $g\cdot s(f)=g\big(s(f\circ \beta(g))\big)$.\\

3. The last statement is clear because the connection on $\pi^*\mathbb{C}l(M|B)\cong \maP \times \mathbb{C}l(2n)$ 
is given by $d+ Ad_*(\bar{\theta})$, where $Ad : PU(2^n) \rightarrow Aut(\mathbb{C}l(2n))$ 
is the representation given through the isomorphism $\Phi$ by conjugation by $SU(2^n)$.
Therefore by definition, we get  
$\Phi\bigg(Ad_*\big(\bar{\theta}\big)\Big(\Phi^{-1}(C)\Big)\bigg)=U_*(\bar{\theta})C - C U_* (\bar{\theta}).$
\end{proof}

As in \cite{MMS1}, the Levi-Civita connection on $M$ also 
induces similarly connections $\nabla^{\mathbb{S}^\pm}$ on $\mathbb{S}^\pm$.
Furthermore, the homomorphism bundle of the vertical 
$Spin$ bundle $\mathbb{S}$ can be identified with $\mathbb{C}l(M|B)$
and recall that it has an extension to $\widehat{\mathbb{C}l(M|B)}=\hat{\maA}$ 
in a neighborhood of the diagonal, and this extended
bundle also has an induced connection, see Proposition \ref{prop.ext.Hom} and \cite{MMS1}.  
We can then define the associated family $\cancel{\partial}_{M|B}$ of projective Dirac operators with kernel 
$$\cancel{\partial}_{M|B}:= cl \cdot \nabla^{\mathbb{S}}_L(\kappa_{\Id}),\quad \kappa_{\Id}=\delta(z-z') \Id_{S}.$$
Here, as in \cite{MMS1}, $\kappa_{\Id}$ is the kernel of the identity operator seen
as a family of projective differential operators on $\mathbb{S}$ 
(i.e. a family of projective pseudodifferential operators with support in the diagonal)
and $\nabla_L^{\mathbb{S}}$ is the connection
restricted to the left variables with $cl$ the contraction given by the Clifford
action of $T(M|B)$ on the left. 
The operator $\cancel{\partial}_{M|B}$ is then odd with respect to the graduation and elliptic with symbol 
$\sigma(\cancel{\partial}_{M|B})(\xi)=cl(\xi)$ the Clifford multiplication.

We now represent the previous family of projective operators 
by a family $\cancel{\partial}_{\maP|B}$ of $SU(2^n)$-transversally elliptic operators. 
Let $\nabla^\pm$ be the $SU(2^n)$-equivariant connections induced by 
the Levi-Civita on the $SU(2^n)$-equivariant vector bundles $\tilde{\mathbb{S}}^\pm$.
We then obtain the corresponding family $\cancel{\partial}_{\maP|B}^+$ of 
$SU(2^n)$-transversally elliptic operators on $\maP$ which is given by
$$ \cancel{\partial}_{\maP|B}^+ = \sum c(e_i) \nabla^+_{e_i},$$
where $c(e_i)$ is the Clifford multiplication introduced above 
and $(e_i)$ is any local orthonormal basis of $T(M|B)$, 
see also \cite{MMS2,Yamashita}. 
Since the principal symbol of $\cancel{\partial}_{\maP|B}^+$ is given by 
$\sigma(\cancel{\partial}_{\maP|B}^+)(\xi)=c(\xi)$ for any $\xi \in T_{SU(2^n)}(\maP|B)$, we get:

\begin{cor}\label{cor.ind.proj.dirac}
The index of the family of projective Dirac operators is given by 
$$\mathrm{Ind}^{M|B}_a(\cancel{\partial}_{M|B}^+)=(2\pi i)^{-n}\int_{M|B} \hat{A}(T(M|B)) \in \maH^{ev}_{dR}(B).$$
\end{cor}

\begin{proof}
Using Theorem \ref{thm.index.proj}, we see that we need to show that 
$$\Ch_{\Id}(\sigma(\cancel{\partial}_{\maP|B}^+))=(2i\pi)^{n}\hat{A}(T(M|B))^{-1} \wedge \operatorname{Thom}(T(M|B)).$$
This follows from the fact that the curvature 
of $\bar{\theta}$ (respectively its image by $U_*$) corresponds to the curvature of $\theta$ 
(respectively to its image by $\mathfrak{so}(2n) \rightarrow \mathfrak{spin}(2n)$) 
and the standard computation of the Chern charater of the symbol of families of Dirac operators. 
First notice that the moment map $\mu^{\nabla}(X)$ vanishes because 
$\mu^{\nabla}(X)=\mathscr{L}^{\tilde{\mathbb{S}}}(X)-\nabla_X=U_*(X)-U_*(\bar{\theta}(X_\maP))=0$.
It follows that $\mu^{\nabla^\pm}(X)=0$ because $\tilde{\omega}_\C$ is $SU(2^n)$-invariant.
Now let $(U_i)_{i\in I}$ be a finite cover of $M$ of trivialisations 
$\mathscr{E}_i : U_i \times SO(2n) \rightarrow \maF=F^{SO}(M|B)$ of the bundle of oriented orthonormal frames of $T(M|B)$
and let $\phi_i : U_i \times PU(2^n) \rightarrow \maP$ be the induced trivialisations of $\maP$.
We shall denote again by $\mathscr{E}_i : U_i \rightarrow \maF$ the section given by $\mathscr{E}_i(x)=\mathscr{E}_i(x,\Id)$ and write 
$\mathscr{E}_i(x,R)=\mathscr{E}_i(x)\circ R=\mathscr{E}_i(x) R$.
Let $(f_i)_{i\in I}$ be a partition of unity subordinated to $(U_i)$. We can then write
$$\Ch_{\Id}^{PU(2^n)}(\sigma(\cancel{\partial}_{\maP|B}^+))(X)=\sum f_i (d\phi_i^{-1})^*(d\phi_i)^*\Ch_e^{PU(2^n)}(\sigma(\cancel{\partial}_{\maP|B}^+))(X),$$
which does not depend on $X\in \mathfrak{su}(2^n)$ since the moment maps vanish. 
Denote by $\tilde{\mathscr{E}}_i : U_i \times \R^{2n} \rightarrow TU_i$ the trivialisation given by 
$\tilde{\mathscr{E}}_i(x,v)=\mathscr{E}_i(x)(v)$.
We have that 
$$(d\phi_i)^* (\sigma(\cancel{\partial}_{\maP|B}^+))(x,w,A)=\hat{A}^{-1}\hat{c}(\mathscr{E}_i^{-1}(x)w) \hat{A}=\hat{A}^{-1} ((\tilde{\mathscr{E}}^{-1}_i)^*\hat{c})(x,w) \hat{A},$$ 
for any $(x,w,A) \in T(U_i|B) \times PU(2^n)$. Now $(\tilde{\mathscr{E}}^{-1}_i)^*\hat{c}$ is 
just the Clifford multiplication for the trivial bundle $T(U_i|B)$ 
acting on the trivial half $Spin$ bundle $U_i \times S^\pm$ and the curvature 
of $\bar{\theta}$ (respectively its image by $U_*$) corresponds to the curvature of $\theta$ 
(respectively to its image by $\mathfrak{so}(2n) \rightarrow \mathfrak{spin}(2n)$) 
therefore using the Chern-Weil isomorphism 
$CW : \maH_{c,PU(2^n)}^\infty(\frak{su}(2^n),T(\maP_{\vert_{ U_i}}|B)) \rightarrow \maH_{c,dR}(T(U_i|B))$, we obtain that 
$$f_i\ \Ch_{\Id}((d\phi_i)^* (\sigma(\cancel{\partial}_{\maP|B}^+))=f_i\ \Ch((\tilde{\mathscr{E}}_i^{-1})^*\hat{c})=f_i\ (2i\pi)^{n}\hat{A}(T(U_i|B))^{-1} \wedge \operatorname{Thom}(T(U_i|B)),
$$
where $\operatorname{Thom}(T(U_i|B))$ is the Thom form of the bundle $T(M|B)\vert_{U_i} \rightarrow U_i$.
Finally, we get 
\begin{align*}
\Ch_{\Id}(\sigma(\cancel{\partial}_{\maP|B}^+))&=\sum f_i\  (2i\pi)^{n}\hat{A}(T(U_i|B))^{-1} \wedge \operatorname{Thom}(T(U_i|B))\\
&=(2i\pi)^{n}\sum f_i\  \big(\hat{A}(T(M|B))^{-1} \wedge \operatorname{Thom}(T(M|B))\big)\vert_{(T(M|B)\vert_{U_i})}\\
&=(2i\pi)^{n}\hat{A}(T(M|B))^{-1} \wedge \operatorname{Thom}(T(M|B)).
\end{align*}
\end{proof}

\begin{remarque}
The equality $\Ch_{\Id}(\sigma(\cancel{\partial}_{\maP|B}^+))=(2i\pi)^{n}\hat{A}(T(M|B))^{-1} \wedge \operatorname{Thom}(T(M|B))$ 
can also be shown by twisting the class $[\sigma(\cancel{\partial}_{\maP|B})]$ by the $SU(2^n)$-equivariant vector bundle 
$\maP \times \tilde{\mathbb{S}}^*$ and looking at the standard formulae since the twisted symbol corresponds to the $SU(2^n)$-equivariant symbol 
of the family of signature operators. 
\end{remarque}

\section{Families of $Spin(2n)$-transversally elliptic Dirac operators }
In this section, we discuss the application of Theorem \ref{thm.main.paradan} 
given by a family of $Spin(2n)$-transversally elliptic operators over the bundle 
of oriented orthonormal frames of an oriented fibration with even dimensional fibres.
The motivation for the study of the index of families of $Spin(2n)$-transversally elliptic Dirac operators 
comes from the index of families of projective Dirac operators. 
Indeed, we will see that the index of families of projective Dirac operators is captured by the index of families of $Spin(2n)$-transversally elliptic Dirac operators. This was already noticed in \cite{Mathai.genera,Paradan:projective}. 
 
Let $p:M\rightarrow B$ be a fibration of compact manifolds as before and 
assume that $T(M|B)$ is oriented and that $\dim (M|B)=\dim M -\dim B=2n$. 
Recall that $B$ is also assumed to be oriented. 
Let $\C l(M|B)$ be the complexified Clifford bundle of $T(M|B)$ as before.
Assume that the metric $g_M$ on $M$ is constructed from the pull back of a metric $g_B$ on $B$ 
and a metric on $T(M|B)$, i.e. $g_M=g_{M|B}\oplus p^*g_B$.

Let $q : \maF=F^{SO}(M|B) \rightarrow M$ be the bundle of oriented orthonormal frames of $T(M|B)$.
Let $\nabla^{LC}$ be the Levi-Civita connection on $M$ and consider the connection
$\nabla^{LCV}=\pi^V\circ \nabla^{LC}$ on $T(M|B)$, 
where $\pi^V : TM \rightarrow T(M|B)$ is the projection induced by the metric on $M$. 
Then $\nabla^{LCV}$ induces the Levi-Civita connection on each fibre $M_b$.
Denote by $\theta$ the induced connection $1$-form on $\maF=F^{SO}(M|B)$ and equip 
$\maF$ with the metric $g_{\maF}=q^*g_M + \langle \theta , \theta \rangle_{\mathfrak{so}(2n)}$, 
where $\langle \cdot , \cdot \rangle_{\mathfrak{so}(2n)}$ is 
a $Ad$-invariant metric on the Lie algebra $\mathfrak{so}(2n)$ of $SO(2n)$.

We have the following trivialisation of $T^*_{SO(2n)}(\maF|B)$, see \cite{Paradan:projective} for the case $B=\{b\}$ 
(i.e. when $B$ is reduced to a point).
\begin{lem}
\begin{enumerate}
\item The map $\alpha : T^*_{SO(2n)}(\maF|B) \rightarrow \maF \times (\R^{2n})^*$ given by 
$\alpha(\mathscr{E},\xi) = (\mathscr{E}, \xi \circ (d_{\mathscr{E}}q)^{-1} \circ \mathscr{E})$
is an $SO(2n)$-equivariant isomorphism. 
We shall denote by $\alpha_{\mathscr{E}} : (T^*_{SO(2n)}\maF)_{\mathscr{E}} \rightarrow (\R^{2n})^*$ the induced map.
\item The map $\alpha$ induces an isomorphism $\tilde{\alpha}$
between the bundle of oriented orthonormal frames of $T_{SO(2n)}(\maF|B)$ and $\maF \times SO(2n)$.
\item The bundle $T_{SO(2n)}(\maF|B)$ has a spin structure. 
\end{enumerate}

\end{lem} 
 
\begin{proof}
(i) Recall that the map $d_{\mathscr{E}}q : T_{SO(2n)}(\maF|B)_{\mathscr{E}} \rightarrow T_{q(\mathscr{E})}(M|B)$ 
is an isomorphism for any $\mathscr{E} \in \maF$. By definition $\mathscr{E}$  
gives an isomorphism $\mathscr{E} : \R^{2n} \rightarrow T_{q(\mathscr{E})}(M|B)$. 
Furthermore, $\alpha$ is $SO(2n)$-equivariant because for any $R\in SO(2n)$ we have,
\begin{align*}
\alpha(\mathscr{E}\circ R , ^tdr_{R^{-1}}\xi)&=(\mathscr{E}\circ R , \xi \circ d_{\mathscr{E}}r_{R}^{-1} \circ (d_{\mathscr{E}\circ R}q)^{-1} \circ \mathscr{E}\circ R) \\
&=(\mathscr{E}\circ R , \xi \circ (d_{\mathscr{E}}(q\circ r_{R}))^{-1} \circ \mathscr{E}\circ R) \\
&=(\mathscr{E}\circ R , \xi \circ (d_{\mathscr{E}}q)^{-1} \circ \mathscr{E}\circ R)\\
&=\alpha(\mathscr{E} , \xi \circ (d_{\mathscr{E}}q)^{-1} \circ \mathscr{E})\cdot R,
\end{align*}
where $r_R$ denotes the right action of $R$.

(ii) This follows directly from (i). Indeed, if $(\mathscr{E}, \mathscr{W})$ is a frame of $T_{SO(2n)}(\maF|B)$ then 
$\tilde{\alpha}(\mathscr{E}, \mathscr{W})=(\mathscr{E}, \mathscr{W}^{-1} \circ (d_{\mathscr{E}}q)^{-1} \circ \mathscr{E})  \in \maF \times  SO(2n)$.

(iii) The map $\maF \times Spin(2n) \rightarrow \maF \times SO(2n)$ defines a spin structure.
\end{proof}

Let us denote by $S^\pm$ the half spinor associated with $Spin(2n)$ 
and  by $F^{Spin}(\maF|B) = \maF \times Spin(2n)$ the bundle of spin frames.

We consider the spin vector bundles 
$\mathbb{S}^\pm(\maF|B)=F^{spin}(\maF|B) \times_{Spin(2n)} S^\pm=\maF \times S^\pm \rightarrow \maF$ 
associated with the spin structure.
The Clifford multiplication $c : T^*_{SO(2n)}(\maF|B) \rightarrow \Hom(S^+,S^-)$ is 
then given by $c(\mathscr{E},\xi)=\hat{c}(\alpha_{\mathscr{E}}(\xi))$, 
where $\hat{c} : (\R^n)^* \rightarrow \mathbb{C}l(2n)$ is the Clifford multiplication  
and this is $Spin(2n)$-equivariant with respect to the action induced by $\zeta : Spin(2n) \rightarrow SO(2n)$. 
In other words, for any $g \in Spin(2n)$ we have 
$c((\mathscr{E},\xi) \cdot g)=c(\mathscr{E}\circ \zeta(g) , ^tdr_{\zeta(g)^{-1}}(\xi))
=\hat{c}(^t\zeta(g)\alpha_{\mathscr{E}}(\xi))=g^{-1}\hat{c}(\alpha_{\mathscr{E}}(\xi))g$, 
where the last equality follows 
from the fact that the inclusion $SO(2n) \hookrightarrow Aut(\mathbb{C}l(2n))=PU(2^n)$ 
is through the spin group, see \cite{MMS1,SpinGeometry}.
Notice that this defines a symbol in $K_{Spin(2n)}(T^*_{Spin(2n)}(\maF|B))=K_{Spin(2n)}(T^*_{SO(2n)}(\maF|B))$.

Let us now introduce the family of $Spin(2n)$-transversally elliptic 
operators associated with the previous symbol.
The connection $\theta$ on $\maF$ induces connections $\nabla^{\pm}$ 
on the half spinor bundles $\mathbb{S}^{\pm}(\maF|B)$, see \cite{SpinGeometry} for instance.  
We then define the family $\cancel{\partial}_{\maF|B}^+$ 
of $Spin(2^n)$-transversally elliptic Dirac operators on $\maF$ by
$$\cancel{\partial}_{\maF|B}^+:=\sum c(e_i)\nabla_{e_i}^+ : C^\infty(\maF,S^+) \rightarrow C^\infty(\maF,S^-),$$  
where $c(e_i)$ is the Clifford multiplication and $(e_i)$ is any local orthonormal basis of $q^*T(M|B)\cong T_{SO(2n)}(\maF|B)$, 
see  \cite{Paradan:projective, Yamashita}.

\begin{thm}\label{thm.ind.dirac}
The distibutional index $\mathrm{Ind}^{\maF|B}_{-\infty}(\cancel{\partial}_{\maF|B}^+)$ of $\cancel{\partial}_{\maF|B}^+$ is given by
$$\mathrm{Ind}^{\maF|B}_{-\infty}(\cancel{\partial}_{\maF|B}^+)=T_{M|B}\ast \delta_{\Id} - T_{M|B} \ast \delta_{-\Id},$$
where $T_{M|B}=(2\pi i )^{-n} \exp_* \int_{M|B} \hat{A}(T(M|B)) \wedge e^{\Theta}$ and $\Theta$ is the curvature of the $SO(2n)$-principal bundle
$\maF \rightarrow M$.
In particular, if $\varphi \in C^\infty(Spin(2n))^{Ad(Spin(2n))}$ is a function equal to $1$ around $\Id$ and $0$ around $-\Id$ then 
$$\langle \mathrm{Ind}^{\maF|B}_{-\infty}(\cancel{\partial}_{\maF|B}^+), \varphi \rangle =(2\pi i)^{-n}\int_{M|B} \hat{A}(T(M|B)) \in \maH^{ev}_{dR}(B).$$
\end{thm}

\begin{proof}
\noindent
Recall the central extension
$$1 \rightarrow \mathbb{Z}_2 \rightarrow Spin(2n) \rightarrow SO(2n) \rightarrow 1.$$
Applying Theorem \ref{thm.main.paradan}, we get that 
 $\mathrm{Ind}^{\maF|B}_{-\infty}(\cancel{\partial}_{\maF|B}^+)= T_{\Id}(\sigma(\cancel{\partial}_{\maP|B}^+))\ast \delta_{\Id} 
 + T_{-\Id}(\sigma(\cancel{\partial}_{\maP|B}^+)\ast \delta_{-\Id}$, where
$$T_\gamma(\sigma(\cancel{\partial}_{\maP|B}^+))=(2i\pi)^{-\dim (M | B)}\exp_*\Big(\int_{T(M|B)|B} \Ch_\gamma(\sigma(\cancel{\partial}_{\maP|B}^+))\wedge\hat{A}(T(M|B))^2\wedge e^\Theta\Big)$$
and where $\Ch_\gamma(\sigma(\cancel{\partial}_{\maP|B}^+))$ is the twisted Chern character, see Definition \ref{def:Chern:twist}.

Since $\Ch_{\Id}(\sigma(\cancel{\partial}_{\maP|B}^+))=-\Ch_{-\Id}(\sigma(\cancel{\partial}_{\maP|B}^+))$, 
we only need to show that 
$$\Ch_{\Id}(\sigma(\cancel{\partial}_{\maP|B}^+))=(2i\pi)^{n}\hat{A}(T(M|B))^{-1} \wedge \operatorname{Thom}(T(M|B)).$$
This can be obtained as in the proof of Corollary \ref{cor.ind.proj.dirac}. 
Let us give an other proof based on the $Spin$ structure, see \cite{BGV,ParadanV:equiThomAndChern}

Recall that the vector bundle $T_{SO(2n)}(\maF|B)=T_{Spin(2n)}(\maF|B)$ is spin therefore using \cite[Proposition 7.43]{BGV} 
$$(2i\pi)^{-n}\Ch_{\Id}^{Spin(2n)}(\sigma(\cancel{\partial}_{\maF|B}^+))(X)
=\hat{A}(T_{SO(2n)}(\maF|B))(X)^{-1} \wedge \operatorname{Thom}(T_{SO(2n)}(\maF|B))(X),$$
where $\operatorname{Thom}(T_{SO(2n)}(\maF|B))(X)$ denotes the Thom form in equivariant cohomology, 
see also \cite{ParadanV:equiThomAndChern}.
Now recalling the identification $T_{SO(2n)}(\maF|B)=q^*T(M|B)$, 
we get 
\begin{align*}
\hat{A}(T_{SO(2n)}(\maF|B))(X)&=q^*\hat{A}(T(M|B))\otimes \phi(X),\qquad \mbox{and},\\
\operatorname{Thom}(T_{SO(2n)}(\maF|B))(X)&=q^*\operatorname{Thom}(T(M|B))\otimes \phi(X),
\end{align*} 
where $\operatorname{Thom}(T(M|B))$ is the Thom form and $\phi$ is equal to $1$ 
on a small neighbourhood of $0\in \mathfrak{spin}(2n)$. In fact, $\phi$ can be taken constant equal to $1$
because here if we equipped $T_{SO(2n)}(\maF|B)$ with the pull-back connection $\pi^*\nabla$ 
of a connection $\nabla$ on $T(M|B)$ then the moment $\mu^{T_{SO(2^n)}(\maF|B)}(X)=0$.
Indeed, here $\mathscr{L}^{T_{SO(2n)}(\maF|B)}(X)$ coincides with $X_\maF$ and $(\pi^*\nabla)_{X}$ coincides also with
$X_\maF$ therefore $\mu^{T_{SO(2n)}(\maF|B)}(X)=\mathscr{L}^{T_{SO(2n)}(\maF|B)}(X)-(\pi^*\nabla)_{X}=0$.
Recall that the image of $\phi$ through the Chern-Weil morphism gives $1$ in cohomology.
Applying the Chern-Weil isomorphism, it follows 
$$\Ch_{\Id}(\sigma(\cancel{\partial}_{\maF|B}^+))=(2i\pi)^{n}\hat{A}(T(M|B))^{-1} \wedge \operatorname{Thom}(T(M|B)).$$
This complete the proof.
\end{proof}
 
 
\begin{remarque}\label{rem.Dirac.proj.spin}
As a corollary, we obtain that the index of a family of projective Dirac operators 
can be computed from the index of the corresponding family of $Spin(2n)$-transversally elliptic Dirac operators on $\maF=F^{SO}(M|B)$.
More precisely, 
 if $\varphi \in C^\infty(Spin(2n))^{Ad(Spin(2n))}$ is a function equal to $1$ around $\Id$ and $0$ around $-\Id$ then 
$$\mathrm{Ind}^{M|B}_a(\cancel{\partial}^+_{M|B})=\langle \mathrm{Ind}^{\maF|B}_{-\infty}(\cancel{\partial}_{\maF|B}^+), \varphi \rangle =(2\pi i)^{-n}\int_{M|B} \hat{A}(T(M|B)) \in \maH^{ev}_{dR}(B).$$
The family $\cancel{\partial}_{\maF|B}^+$ of $Spin(2n)$-transversally elliptic operators  
induces in some sense the family of $SU(2^n)$-transversally elliptic operators.
Indeed, the connection on $\mathbb{S}(\maF|B)$ is given by $\nabla=d+\zeta_*^{-1}(\theta)$, 
where $\zeta_* : \mathfrak{spin}(2n) \rightarrow \mathfrak{so}(2n)$. 
If we consider the push forward map $q_*$ defined as in Theorem \ref{thm.MMS2.projective.trans}, 
but now replacing $\pi : \maP\rightarrow M$ by $q : \maF \rightarrow M$ 
and the central extension of $PU(2^n)$ by the central extension of $SO(2n)$, 
then we recover the family $\cancel{\partial}_{M|B}^+$ of projective Dirac operators from 
the family $\cancel{\partial}_{\maF|B}^+$ of $Spin(2n)$-transversally elliptic operators. 
In other words, the pull back to $\maF$ of the family $\cancel{\partial}_{M|B}^+$ of projective Dirac operators is 
the family $\cancel{\partial}_{\maF|B}^+$ of $Spin(2n)$-transversally elliptic operators, 
see also \cite{Mathai.genera,Paradan:projective}.       
\end{remarque}

 \subsection{Example}
We now discuss an example. Let us consider the complex projective plan $ \mathbb{CP}^2$ equipped 
 with the Fubini-Study metric $g^{FS}$. Recall that the Fubini-Study metric is given in affine 
 charts using complex coordinates $(z_1,\bar{z}_1,z_2, \bar{z}_2)$  by
 $$g_{i\bar{j}}^{FS}=g^{FS}(\partial_{z_i},\partial_{\bar{z}_j}) = \frac{(1+|z|^2)\delta_{i\bar{j}} -\bar{z_i}z_j}{(1+|z|^2)^2},$$
 and $g_{ij}^{FS}=g^{FS} (\partial_{z_i},\partial_{z_j})=0$.
We let $S^1$ act isometrically on $\mathbb{CP}^2$ by $\gamma\cdot[z_0,z_1,z_2]=[z_0,z_1,\gamma z_2]$.
Let $S^5 \subset \C^3$ be the unit sphere equipped with the induced metric $g^{S^5}$ from $\C^3$.
Recall the principal $S^1$-bundle $q : S^5 \rightarrow \mathbb{CP}^2$ 
given by the orbits of the diagonal action of $S^1$ on $S^5 \subset \C^3$.
By letting $S^1$ act diagonally on $S^5 \times \mathbb{CP}^2$ 
by $\gamma \cdot ((w_0,w_1,w_2),[z_0,z_1,z_2])=((\gamma w_0,\gamma w_1,\gamma w_2),[z_0,z_1,\gamma z_3])$, 
we obtain a principal $S^1$-bundle $\rho : S^5 \times \mathbb{CP}^2 \rightarrow M:= S^5\times_{S^1} \mathbb{CP}^2$.
Since the metric $g^{S^5 \times \mathbb{CP}^2}= g^{S^5} + g^{\mathbb{CP}^2}$ on $S^5 \times \mathbb{CP}^2$ is $S^1$ invariant, 
we obtain a metric on $M$. Let us now consider the fibration
$p: M:=S^5 \times_{S^1} \mathbb{CP}^2 \rightarrow \mathbb{CP}^2=B$ 
given by $p([x,[z]])=q(x)$. It is then clear that $p$ is also a Riemannian fibration 
since it is obtained from the Riemannian fibration 
$S^5 \times \mathbb{CP}^2 \rightarrow \mathbb{CP}^2$ 
with projection $(x,[z]) \mapsto q(x)=q\circ p_1(x,[z])$.

The vertical tangent bundle $T(M|B)=\ker dp$ is given by $S^5 \times_{S^1} T\mathbb{CP}^2$ 
and the bundle $F^{SO}(M|B)$ of oriented orthonormal frames of $T(M|B)$ is given by $\maF:=F^{SO}(M|B)=S^5 \times_{S^1} F^{SO}(\mathbb{CP}^2)$, 
where $F^{SO}(\mathbb{CP}^2)$ is the bundle of oriented orthonormal frames of $\mathbb{CP}^2$.
Let $\mathbb{C}l(\mathbb{CP}^2)$ be the complexified Clifford bundle 
of $\mathbb{CP}^2$.
Similarly, let $\mathbb{C}l(M|B)$ be the Clifford bundle of $T(M|B)$ 
then $\mathbb{C}l(M|B)=S^5 \times_{S^1} \mathbb{C}l(\mathbb{CP}^2)$.
 

Notice that $\mathbb{CP}^2$ can be seen as a fibration of oriented Riemannian manifolds of dimension $2\times 2$ over the point.
Denote the $Spin(4)$-transversally elliptic Dirac 
$\cancel{\partial}_{F^{SO}(\mathbb{CP}^2)}$ associated over the bundle of oriented orthonormal fame.
This allows to define a family $\mathrm{Id}_{S^5} \otimes \cancel{\partial}_{F^{SO}(\mathbb{CP}^2)}$ of differential operators along the fibers of $S^5\times \mathbb{CP}^2 \rightarrow S^5$, see \cite{Atiyah-Singer:IV}.

\begin{lem}
The family $\mathrm{Id}_{S^5} \otimes \cancel{\partial}_{F^{SO}(\mathbb{CP}^2)}$  on $S^5 \times F^{SO}(\mathbb{CP}^2) \rightarrow S^5$ 
is a family of $Spin(4)$-transversally elliptic operators which is $S^1$-invariant with respect to the diagonal $S^1$-action introduced above. 
\end{lem}

\begin{proof}

Indeed, recall the trivialisation $\alpha : T_{SO(2n)} (F^{SO(2n)}(\mathbb{CP}^2) \rightarrow F^{SO(2n)}(\mathbb{CP}^2) \times \R^4$ given by
$\alpha(\mathscr{E},w)=(\mathscr{E}, \mathscr{E}^{-1}\circ d_{\mathscr{E}}q (w))$.
This induces the trivialisation $\tilde{\alpha} : F^{SO}(T_{SO(2n)}F^{SO}(\mathbb{CP}^2)) \rightarrow F^{SO}(\mathbb{CP}^2) \times SO(4)$ given by
$\tilde{\alpha}( \mathscr{E},\mathscr{W})=(\mathscr{E}, \mathscr{E}^{-1}\circ d_{\mathscr{E}}q\circ \mathscr{W})$.
Let us look at the action of $S^1$ through the trivialisations $\alpha$ and $\tilde \alpha$.
The action on $T_{SO(2n)} (F^{SO(2n)}(\mathbb{CP}^2))$ is given by 
$e^{i\theta} \cdot (\mathscr{E}, w) = (d_{q(\mathscr{E})}e^{i\theta} \circ \mathscr{E} , d_{\mathscr{E}}(d_{q(\mathscr{E})}e^{i\theta} ) (w))$.
We compute then
\begin{align*}
\alpha(e^{i\theta} \cdot (\mathscr{E}, w)) &= \alpha(d_{q(\mathscr{E})}e^{i\theta} \circ \mathscr{E} , d_{\mathscr{E}}(d_{q(\mathscr{E})}e^{i\theta} ) (w))\\
&=(d_{q(\mathscr{E})}e^{i\theta} \circ \mathscr{E} , \mathscr{E}^{-1} \circ (d_{q(\mathscr{E})}e^{i\theta})^{-1} 
\circ \big[d_{[d_{q(\mathscr{E})}e^{i\theta}\circ \mathscr{E}]}q \big] \circ d_{\mathscr{E}}(d_{q(\mathscr{E})}e^{i\theta} ) (w))\\
&=(d_{q(\mathscr{E})}e^{i\theta} \circ \mathscr{E} , \mathscr{E}^{-1} \circ (d_{e^{i\theta}q(\mathscr{E})}e^{-i\theta}) 
\circ d_{ \mathscr{E}}(q \circ d_{q(\mathscr{E})}e^{i\theta} ) (w)).
\end{align*}

\noindent
Since $q\circ d_{\mathscr{E}}e^{i\theta}=e^{i\theta} \cdot q$, in other words since $q$ is $S^1$-equivariant, we obtain
\begin{align*}
\alpha(e^{i\theta} \cdot (\mathscr{E}, w)) 
&=(d_{q(\mathscr{E})}e^{i\theta} \circ \mathscr{E} , \mathscr{E}^{-1} \circ (d_{e^{i\theta}q(\mathscr{E})}e^{-i\theta}) 
\circ d_{q(\mathscr{E})}e^{i\theta} \circ d_{ \mathscr{E}}q  (w))\\
&=(d_{q(\mathscr{E})}e^{i\theta} \circ \mathscr{E} , \mathscr{E}^{-1} \circ d_{ \mathscr{E}}q  (w)).
\end{align*} 

\noindent
Therefore, the induced action on $F^{SO}(\mathbb{CP}^2) \times SO(4)$ is given by 
$$e^{i\theta} \cdot (\mathscr{E},\mathscr{W})=(d_{q(\mathscr{E})}e^{i\theta} \circ \mathscr{E} , \mathscr{E}^{-1} \circ d_{\mathscr{E}}q \circ \mathscr{W})$$
and consequently the actions on $F^{SO}(\mathbb{CP}^2) \times Spin(2n)$ and $F^{SO}(\mathbb{CP}^2) \times S^{\pm}$ are given by the same formula, 
that is the action is through the action on the first component $F^{SO}(\mathbb{CP}^2)$.
It follows then also that $\forall (\mathscr{E},w) \in T_{SO(4)}F^{SO}(\mathbb{CP}^2)$, we have
$$c(e^{i\theta} \cdot (\mathscr{E},w))=\hat{c}(\mathscr{E}^{-1} \circ d_{\mathscr{E}}q(w)).$$

\noindent
This shows that the operator is $S^1$-invariant because 
$\nabla^+_{e_j(q(\mathscr{E}))}(s\circ de^{i\theta} )(\mathscr{E})
=\nabla_{d_{q(\mathscr{E})}e^{i\theta} (e_j(q(\mathscr{E}))}(s) (d_{q(\mathscr{E})}e^{i\theta}\circ \mathscr{E})$.

We then get the result because the vertical transverse symbol of $\mathrm{Id}_{S^5} \otimes \cancel{\partial}_{F^{SO}(\mathbb{CP}^2)}$ which is
given $\forall (x,\mathscr{E},w)\in S^5 \times T_{SO(2n)}F^{SO}(\mathbb{CP}^2)$ by 
$$\sigma(\mathrm{Id}_{S^5} \otimes \cancel{\partial}_{F^{SO}(\mathbb{CP}^2)})(\mathscr{E},w)=c(\alpha(\mathscr{E},w))$$
is invertible for non zero $w$.
\end{proof}

By restriction to $S^1$-invariant functions, we obtain a family over $B=S^5 / S^1 =\mathbb{CP}^2$ 
of $Spin(4)$-transversally elliptic operators
$$\cancel{\partial}_{\maF|B}^+ : C^{\infty}(\maF,\mathbb{S}^+) \rightarrow C^\infty(\maF,\mathbb{S}^-). $$

Let $\upsilon=\frac{i}{2} \sum_{i=0}^2 Z_kd\bar{Z}_k -\bar{Z}_k dZ_k$ be the standard 
$S^1$-connection on the principal bundle $S^5 \rightarrow \mathbb{CP}^2$
associated with the Fubini-Study metric on $\mathbb{CP}^2$. Here $(Z_0,Z_1,Z_2)\in S^5 \subset \C^3$. 
Denote by $\Upsilon$ the curvature of $\upsilon$. 
Recall that $\Upsilon$ is given in trivialisation corresponding to affine charts 
$U_i=\{Z_i\neq 0\}$ of $\mathbb{CP}^2$ with coordinates $z=(z_1,z_2)$ by
$\Upsilon = \frac{i}{2} \frac{(1+|z|^2)^2\sum dz_j\wedge d\bar{z}_j - \sum \bar{z}_j dz_j \wedge \sum z_k d\bar{z}_k }{(1+|z|^2)^2}$ 
and is therefore a $2$-form with real coefficients.
 
\begin{prop}
Let $\cancel{\partial}_{M|B}^+$ be the 
family of projective Dirac operators corresponding to $\cancel{\partial}_{\maF|B}^+$, 
see Section \ref{section.fam.proj.Dirac} and Remark \ref{rem.Dirac.proj.spin}.
We have
\begin{equation*}
\mathrm{Ind}^{M|B}_a(\cancel{\partial}_{M|B}^+)=\langle \mathrm{Ind}^{\maF|B}_{-\infty}(\cancel{\partial}_{\maF|B}^+), \varphi\rangle =-\frac{1}{8} - \frac{\Upsilon}{2^7\cdot 3^2}\frac{133}{15}
\end{equation*}
for any $\varphi \in C^\infty(Spin(4))^{Ad(Spin(4))}$ equal to $1$ around $\Id$ and $0$ around $-\Id$.
\end{prop} 
 
\begin{proof}
Recall that the $\hat{A}$-genus $\hat{A}(E,\nabla)=\det^{1/2}(\frac{F/2}{\sinh(F/2)})$ of a real vector bundle $E \rightarrow Z$
with connection $\nabla$ and curvature $F$
over a manifold $Z$ belongs to $\maA^{4\bullet}(Z,\R)$. Furthermore, 
it is the form associated with the power series $h(t)=\frac{1}{2}\ln(\frac{t/2}{\sinh(t/2)})$, that is
\begin{equation*}
\hat{A}(E,\nabla)=\exp \Tr\bigg(\frac{1}{2}\ln\big(\frac{F/2}{\sinh(F/2)}\big)\bigg),
\end{equation*}
see \cite[Section 1.5]{BGV}.
Recall that $M=S^5 \times_{S^1} \mathbb{CP}^2$ is $8$-dimensional and 
therefore its $\hat{A}$-genus only has non vanishing terms of degree $0$, $4$ and $8$. 
We then have using the expansion series of $f(t)$  that 
\begin{equation*}\label{eq.exp.A.genus}
\hat{A}(T(M|B))=1-\frac{1}{2^2\cdot 12} \Tr((F^{T(M|B)})^2)+\frac{1}{2^4 \cdot 360}\Tr((F^{T(M|B)})^4) + \frac{1}{2^4\cdot 288} \Tr((F^{T(M|B)})^2)^2,
\end{equation*}
since higher coefficients vanish because their degree is bigger than $\dim M=8$.

\noindent
Using Theorem \ref{thm.ind.dirac}, we get that
\begin{equation*}
\langle\mathrm{Ind}^{\maF|B}_{-\infty}(\cancel{\partial}_{\maF|B}^+), \varphi \rangle =-(2\pi)^{-2}\int_{M|B} \hat{A}(T(M|B)).
\end{equation*}
Now recall that integration along the fibers commutes with the Chern-Weil morphism, 
see \cite[Proposition 7.35]{BGV} for instance. Let then $\nabla^{LC}$ be the Levi-Civita connection on $T\mathbb{CP}^2$ 
and denote by $F$ its curvature.
We now consider the $S^1$-equivariant curvature of $\nabla^{LC}$ given by $F(X)=F+\mu(X)$, 
where $\mu(X)=\maL^{T\mathbb{CP}^2}(X)-\nabla^{LC}_X=-X\nabla^{LC}1_{\mathbb{CP}^2}$ is the moment associated with $X\in \R=Lie(S^1)$,
where the last equality comes from the vanishing of the torsion of $\nabla^{LC}$, see \cite[Example 7.8]{BGV}. 
It is shown in \cite[Lemma 7.37]{BGV} that this data defines a curvature $2$-form $F^{T(M|B)}$ on $T(M|B)$ given by $F^{T(M|B)}=F+\Upsilon\mu(1)$, 
where $1$ is the basis of $Lie(S^1)$.

\noindent
We shall denote by $[f \otimes\alpha]_{\max}=f \otimes [\alpha]_{\max}$ the  component of 
$f\otimes \alpha \in C^\infty(\R)\otimes \maA(\mathbb{CP}^2)$ with maximal degree 
(with respect to the form degree on $\maA(\mathbb{CP}^2)$).  We have
\begin{align*}
[\Tr((F(X))^2)]_{\max}&=[\Tr((F+X \mu(1))^2)]_{\max}=\Tr(F^2)
\end{align*}
and
\begin{align*}
[\Tr(F(X)^4)]_{\max}&=[\Tr((F^2 + X F\mu(1) +X\mu(1) F + X^2\mu(1))^2)]_{\max}\\
&=X^2\Tr(F^2\mu(1)^2)+X^2\Tr\bigg((F\mu(1))^2\bigg)    +X^2\Tr\bigg(F\mu(1)^2F\bigg)   \\
&\qquad   + X^2\Tr\bigg((\mu(1)F)^2\bigg)  +X^2\Tr\bigg(\mu(1)F^2\mu(1)\bigg)    +X^2\Tr(\mu(1)^2F^2)\\
&=X^2\bigg(4\Tr(F^2\mu(1)^2)+2\Tr((F\mu(1))^2)\bigg)
\end{align*}
and 
\begin{align*}
[\Tr(F(X)^2)^2]_{\max}=X^2\bigg(2\Tr(F^2)\Tr(\mu(1)^2) + 4\Tr(F\mu(1))^2\bigg).
\end{align*}
We did the computations with the help of the free software Sagemath in the coordinate chart $U_0=\{z_0\neq 0\}$ with coordinates $(z_1,z_2)\in \C^2$. 
In the sequel $|z|^2=|z_1|^2+|z_2|^2$. 
We obtained that the square of the curvature $F$ is given by the diagonal matrix
\begin{equation*}
F^2=\mathrm{diag}\bigg(\frac{3}{(1+|z|^2)^3}dz_1 \wedge d\bar{z_1} \wedge dz_2 \wedge d\bar{z}_2\bigg),
\end{equation*}
and therefore $\Tr(F^2)=\frac{12}{(1+|z|^2)^3}dz_1 \wedge d\bar{z_1} \wedge dz_2 \wedge d\bar{z}_2$.
We can then compute 
\begin{align*}
\frac{(2i\pi)^{-2}}{-48}\int_{\mathbb{CP}^2} \Tr(F^2) &= \int_{U_0}\frac{12}{(1+|z|^2)^3}dz_1 \wedge d\bar{z_1} \wedge dz_2 \wedge d\bar{z}_2\\
&=\frac{1}{48}\int_{\R^4} -4 \frac{12}{(1+|z|^2)^3} dx_1dy_1dx_2dy_2.
\end{align*}
We chose to compute this integral using polar coordinates for $(x_i,y_i)=r_i(\cos(\theta_i),\sin(\theta_i))$ 
and then again polar coordinates with respect to $(r_1,r_2)$
because the other integrals seem to be easier to compute using this choice.
We get 
\begin{align*}
\frac{(2i\pi)^{-2}}{-48}\int_{\mathbb{CP}^2} \Tr(F^2)&=-\int_{\mathbb{R}_+^2} \frac{r_1r_2}{(1+r_1^2+r_2^2)^3}dr_1dr_2\\ 
&=-\int_{\R_+}\int_{0}^{\pi/2} \frac{r^3\cos(\theta)\sin(\theta)}{(1+r^2)^3} d\theta dr\\
&=-\bigg[\frac{\sin(\theta)^2}{2}\bigg]_0^{\pi/2}\bigg[\frac{r^4}{4(1+r^2)^2}\bigg]_0^\infty\\
&=-\frac{1}{8}.
\end{align*}

\noindent
Since $F^2$ is diagonal, we obtain that  
$$\Tr(F^2\mu(1)^2)=\frac{3\Tr(\mu(1)^2)}{(1+|z|^2)^3}dz_1 \wedge d\bar{z_1} \wedge dz_2 \wedge d\bar{z}_2.$$
We have
$$\Tr(\mu(1)^2)=-2 \frac{|z_1|^4+2|z_2|^4-2|z_1|^2|z_2|^2-2|z_2|^2+2|z_1|^2+1}{(1+|z|^2)^2}.$$
We can then compute as we did before
\begin{align*}
\int_{\mathbb{CP}^2} \Tr(F^2\mu(1)^2) &=\int_{U_0} \frac{3\Tr(\mu(1)^2)}{(1+|z|^2)^3}dz_1 \wedge d\bar{z_1} \wedge dz_2 \wedge d\bar{z}_2\\
&=(2\pi)^2\frac{3}{2}\\
&=\frac{1}{4}\int_{\mathbb{CP}^2} \Tr(F^2)\Tr(\mu(1)^2).
\end{align*}
We have
\begin{equation*}
\Tr(F\mu(1))^2=-8 \frac{2|z_1|^4+9|z_2|^4-9(|z_1|^2+1)|z_2|^2+4|z_1|^2+2}{(1+|z|^2)^5}dz_1\wedge d\bar{z}_1\wedge dz_2 \wedge d\bar{z}_2.
\end{equation*}
Therefore, we get
\begin{equation*}
\int_{\mathbb{CP}^2} \Tr(F\mu(1))^2=4(2\pi)^2.
\end{equation*}
For the last needed term, we obtain
\begin{equation*}
\Tr\big((F\mu(1))^2\big)=-\frac{8|z_1|^4+9|z_2|^4-13|z_1|^2|z_2|^2-13|z_2|^2+16|z_1|^2+8}{(1+|z|^2)^5}
\end{equation*}
and therefore
\begin{equation*}
\int_{\mathbb{CP}^2} \Tr\big((F\mu(1))^2\big)=\frac{5(2\pi)^2}{3}.
\end{equation*}

\noindent
Putting everything together, we finally get
\begin{align*}
\mathrm{Ind}^{M|B}_a(\cancel{\partial}_{M|B}^+ )
&= -\frac{1}{8} + \frac{(2i\pi)^{-2}\Upsilon^2}{2^4\cdot 2^3\cdot 3^2} \Bigg[ \frac{4}{5}\int_{\mathbb{CP}^2} \Tr(F^2\mu(1)^2) 
		+ \frac{2}{5} \int_{\mathbb{CP}^2} \Tr\big((F\mu(1))^2\big) \\ 
& \hspace*{6cm}+\frac{1}{2} \int_{\mathbb{CP}^2} \Tr(F^2)\Tr(\mu(1)^2) 
						+\int_{\mathbb{CP}^2}\bigg(\Tr(F\mu(1))\bigg)^2 \Bigg] \\
&=-\frac{1}{8} - \frac{\Upsilon}{2^4\cdot 2^3\cdot 3^2} \Bigg[\frac{6}{5} + \frac{10}{15} + 3 + 4 \Bigg]\\
&=-\frac{1}{8} - \frac{\Upsilon}{2^7\cdot 3^2}\frac{133}{15}.
\end{align*}
\end{proof} 
 
\medskip

\noindent
{\bf{Acknowledgements.}} I would like to express my gratitude to M.-T. Benameur and V. Nistor for useful discussions,
suggestions and encouragement during the redaction of this paper.
I would also like to thank R. Côme, P. Carrillo-Rouse, M. Lesch, H. Oyono-Oyono,  
P.-E. Paradan, M. Puschnigg, Y. Sanchez Sanchez and E. Schrohe 
for many helpful discussions.

\footnotesize
\bibliographystyle{plain}
\bibliography{Transversalement_elliptique}

\end{document}